\documentclass[11pt]{amsart}
\usepackage{amsmath,amssymb}
\usepackage{fullpage}
\usepackage{psfrag}  
\usepackage{hyperref}
\usepackage{amsfonts,amsthm,url,tabularx,array,mathtools,enumerate,verbatim,float} 
\usepackage{tikz} 
\usepackage{caption}
\usepackage{enumitem}
\usepackage{subcaption}
\usetikzlibrary{topaths,calc} 
\tikzstyle{vertex} = [fill,shape=circle,node distance=80pt]
\tikzstyle{edge} = [opacity=0.4,fill opacity=0.0,line cap=round, line join=round, line width=40pt]
\tikzstyle{elabel} =  [fill,shape=circle,node distance=30pt]
\pgfdeclarelayer{background}
\pgfsetlayers{background,main}
\usepackage{thmtools}
\def\COMMENT#1{}
\let\COMMENT=\footnote 
\def\TASK#1{}

\newdimen\margin   
\def\textno#1&#2\par{%
    \margin=\hsize
    \advance\margin by -4\parindent
           \setbox1=\hbox{\sl#1}%
    \ifdim\wd1 < \margin
       $$\box1\eqno#2$$%
    \else
       \bigbreak
       \hbox to \hsize{\indent$\vcenter{\advance\hsize by -3\parindent
       \sl\noindent#1}\hfil#2$}%
       \bigbreak
    \fi}

\def\eps{\varepsilon}
\def\ex{\mathbb{E}}
\def\prob{\mathbb{P}}

\def\rank{\text{rank}}
\newcommand{\polylog}{{\rm polylog}}
\DeclarePairedDelimiter\floor{\lfloor}{\rfloor} 
\DeclarePairedDelimiter\ceil{\lceil}{\rceil}  

\newtheorem{thm}{Theorem}

\newtheorem{lemma}[thm]{Lemma}
\newtheorem{prop}[thm]{Proposition}

\newtheorem{claim}[thm]{Claim}

\newcommand{\msc}[1]{\begin{center}MSC2010: #1.\end{center}}

\declaretheoremstyle[notefont=\bfseries,notebraces={}{},%
    headpunct={},postheadspace=1em,bodyfont=\normalfont\itshape]{mystyle}
\declaretheorem[style=mystyle,numbered=no,name=Theorem]{thm-hand}
\declaretheorem[style=mystyle,numbered=no,name=Proposition]{prop-hand}
\declaretheorem[style=mystyle,numbered=no,name=Claim]{claim-hand}
\declaretheorem[style=mystyle,numbered=no,name=Fact]{fact-hand}

\allowdisplaybreaks

\usetikzlibrary{arrows.meta,calc,fit,positioning,
                shapes.geometric,shapes.misc}

\pgfdeclarelayer{foreground} 
\pgfdeclarelayer{background}
   \pgfsetlayers{background,%
                 main,%
                 foreground%
                 }

\title{The Maker-Breaker Rado game on a random set of integers}
\author{Robert Hancock}
\begin{document}
\label{firstpage}
\begin{abstract} 
Given an integer-valued matrix $A$ of dimension $\ell \times k$ and an integer-valued vector $b$ of dimension $\ell$, the Maker-Breaker $(A,b)$-game on a set of integers $X$ is the game where Maker and Breaker take turns claiming previously unclaimed integers from $X$, and Maker's aim is to obtain a solution to the system $Ax=b$, whereas Breaker's aim is to prevent this. 
When $X$ is a random subset of $\{1,\dots,n\}$ where each number is included with probability $p$ independently of all others, we determine the threshold probability $p_0$ for when the game is Maker or Breaker's win, for a large class of matrices and vectors. This class includes but is not limited to all pairs $(A,b)$ for which $Ax=b$ corresponds to a single linear equation. 
The Maker's win statement also extends to a much wider class of matrices which include those which satisfy Rado's partition theorem. 
\end{abstract}
\date{\today}
\maketitle

\msc{91A24, 11B75, 05C57}

\section{Introduction}\label{secintro}
Given a finite set $X$ and a family of subsets of $X$, $\mathcal{F} \subseteq \mathcal{P}(X)$, we define the Maker-Breaker game on $(X,\mathcal{F})$ to be the game where Maker and Breaker take turns to select a previously unchosen element $x \in X$, and at the conclusion of the game, if Maker has claimed all of the elements of some $F \in \mathcal{F}$, then Maker wins. Otherwise Breaker has claimed at least one element $x$ in every set $F \subseteq \mathcal{F}$, and Breaker wins. The set $X$ is known as the \emph{board}, and the family $\mathcal{F}$ as the \emph{winning sets}. If Maker has a strategy so that no matter how Breaker plays, Maker can always win, then we call the game \emph{Maker's win}. If Maker has no such strategy, then since there is no draw scenario, the game is \emph{Breaker's win}.

Maker-Breaker games first stemmed from a seminal paper by Erd\H{o}s and Selfridge~\cite{ES}, where they proved their famous criterion which gives a general winning strategy for Breaker. Some well-known examples of Maker-Breaker games are where the board $X$ is the edge set of a complete graph $K_n$, and the winning sets $\mathcal{F}$ are all sets of edges which correspond to a perfect matching; a Hamilton cycle; or a fixed subgraph $H$. All of these games turn out to be Maker's win if $n$ is sufficiently large, therefore an adjustment to the game is required if we wish to make the problem of determining whose win the game is more interesting. 
This leads to the following two variations of Maker-Breaker board games, which have each received significant attention.
\begin{itemize}
\item{{\bf Biased board games.} Maker claims one element of the board per turn, whereas Breaker claims $b$ elements per turn, for some fixed $b \in \mathbb{N}$. We call the game the \emph{$(1:b)$ game on $(X, \mathcal{F})$}. Maker-Breaker games are `bias-monotone' (see e.g.~\cite{HKSS}).  This means that there exists a \emph{threshold bias} $b_0$ such that the $(1:b)$ game on $(X,\mathcal{F})$ is Maker's win if and only if $b<b_0$.}
\item{{\bf Random board games.} For a fixed probability $p$ and game $(X,\mathcal{F})$, let $X_p$ be obtained by including each element $x \in X$ with probability $p$ independently of all other elements, and let $\mathcal{F}_p:=\{F \in \mathcal{F}: x \in X_p \text{ for all } x \in F\}$. We then consider the game on the random board $(X_p,\mathcal{F}_p)$, noting that it is a probability space of games. By the monotonicity of the game $(X, \mathcal{F})$ being Maker's win, the existence of a threshold function follows from~\cite{BT}. That is, there exists a \emph{threshold probability} $p_0$ such that $$\lim_{n \to \infty} \mathbb{P} [\text{The game on $(X_p,\mathcal{F}_p)$ is Maker's win} ] = \begin{cases} 1 & \text{if $p \gg p_0$;} \\ 0 & \text{if $p \ll p_0$.} \end{cases}$$}
\end{itemize}
The interesting problem now is to determine the threshold bias and threshold probability for various Maker-Breaker games. For examples and further history of combinatorial board games, see e.g.~\cite{Beck1,HKSS}.

Random board games were first introduced by Stojakovi\'c and Szab\'o~\cite{SS}, who considered games played on a random subset of the edges of a complete graph. Note that this precisely corresponds to the edges of the Erd\H os--R\'enyi random graph $G_{n,p}$. Here, we focus on the game where Maker's aim is to obtain a solution to a system of linear equations within a random set of integers. To be precise, in our Maker-Breaker game, the board will be a random subset of $[n]:=\{1,\dots,n\}$, that is $[n]_p$, which is obtained by including each element of $[n]$ with probability $p$ independently of all other elements. The winning sets are all sets of size $k$ which correspond to a $k$-distinct solution (i.e. $x=(x_1,\dots,x_k)$ has each $x_i$ distinct) to a system of linear equations $Ax=b$, where $A$ is a fixed integer-valued matrix of dimension $\ell \times k$ and $b$ is a fixed integer-valued vector of dimension $\ell$. We call such a game played on a set of integers $X$ the \emph{$(A,b)$-game on $X$}. The class of all $(A,b)$-games are known as \emph{Rado games} (introduced in~\cite{KRSS}), due to the intimate link with Rado's partition theorem, which will be discussed shortly. 

Maker-Breaker games in this setting were first considered by Beck~\cite{Beck2}, who studied the \emph{van der Waerden game}. Here, Maker's aim is to obtain a $k$-term arithmetic progression $a,a+r,\dots,a+(k-1)r$ for some $a,r \in \mathbb{N}$ and fixed $k \in \mathbb{N}$. Note that the set of $k$-term arithmetic progressions in $[n]$ exactly coincides with the set of $k$-distinct solutions to $Ax=0$ in $[n]$ where $A$ is the $(k-2) \times k$ matrix given by $$\begin{pmatrix} 1 & -2 & 1 & 0 & \cdots & 0 & 0 & 0 \\ 0 & 1 & -2 & 1 & \cdots & 0 & 0 & 0 \\ & & & & \ddots & & & \\ 0 & 0 & 0 & 0 & \cdots & 1 & -2 & 1  \end{pmatrix}.$$ Beck determined that the smallest $n \in \mathbb{N}$ such that the $(A,0)$-game on $[n]$ is Maker's win is $n=2^{k(1+o(1))}$.

Here we consider a generalisation of the van der Waerden game using the following definitions. 
Let $A$ be a fixed integer-valued matrix of dimension $\ell \times k$ and $b$ a fixed integer-valued vector of dimension $\ell$.
We call a pair $(A,b)$ (and the matrix $A$ in the case where $b=0$) \emph{irredundant} if there exists a $k$-distinct solution to $Ax=b$ in $\mathbb{N}$, and \emph{partition regular} if for any finite colouring of $\mathbb{N}$, there is always a monochromatic solution (of any kind) to $Ax=b$. A cornerstone result in the area of Ramsey theory for integers is Rado's theorem~\cite{Rado}, which characterises all partition regular pairs $(A,b)$. In~\cite{HL}, Hindman and Leader extended Rado's theorem to characterise all pairs $(A,b)$ for which given any finite colouring of $\mathbb{N}$, there is always a monochromatic $k$-distinct solution to $Ax=b$ (in particular, if $b=0$ then $A$ must be irredundant and partition regular). Hindman and Leader's result implies that given such a pair $(A,b)$, 
if $n$ is sufficiently large then however we $2$-colour $[n]$, there exists a monochromatic $k$-distinct solution to $Ax=b$.  
So in order for Breaker to win the $(A,b)$-game on $[n]$, he must himself obtain a $k$-distinct solution. But then by the classic strategy-stealing argument, Maker can claim this solution for herself. Thus this game is easily shown to be Maker's win. Therefore it is interesting to consider biased and random versions of the $(A,b)$-game on $[n]$. In a very recent paper of Kusch, Ru\'e, Spiegel and Szab\'o~\cite{KRSS}, the biased version is considered. In this paper, we consider the random version. 

In fact (as in~\cite{KRSS}), we consider a wider class of pairs $(A,b)$. Let $(*)$ be the following matrix property: 
\begin{itemize}
\item[$(*)$] Under Gaussian elimination $A$ does not have any row which consists of precisely two non-zero rational entries. 
\end{itemize} 

An equivalent definition, \emph{abundant}, is given in~\cite{KRSS}. 
Note that in~\cite{HST} it is proven that irredundant partition regular matrices are a strict subclass of irredundant matrices which satisfy $(*)$. 

We need the following definitions. For a matrix $A$ index the columns by $[k]$. For a partition $W \dot \cup \overline{W} = [k]$ of the columns of $A$, let $A_{\overline{W}}$ be the matrix obtained from $A$ by restricting to the columns indexed by $\overline{W}$. Let $\rank (A_{\overline{W}})$ be the rank of $A_{\overline{W}}$, where $\rank (A_{\overline{W}})=0$ if $\overline{W} = \emptyset$. Then define 
$$m(A):= \max_{\stackrel{W \dot \cup \overline{W} = [k]}{|W| \geq 2}} \frac{|W|-1}{|W|-1+\rank (A_{\overline{W}})-\rank (A)}.$$
The biased game result of~\cite{KRSS} is the following.
\begin{thm}[\cite{KRSS}, Theorem 1.4 and Proposition 1.5]\label{KRSSthm}
Let $A$ be a fixed integer-valued matrix of dimension $\ell \times k$ and $b$ a fixed integer-valued vector of dimension $\ell$. Given the pair $(A,b)$ and the matrix $A$ are both irredundant, then we have the following:
\begin{itemize}
\item[(i)]{If $A$ satisfies $(*)$ then the threshold bias for the $(A,b)$-game on $[n]$ is $\Theta(n^{1/m(A)})$;}
\item[(ii)]{If $A$ does not satisfy $(*)$ then the $(1:2)$ $(A,b)$-game on $[n]$ is Breaker's win.}
\end{itemize}
\end{thm} 

In this paper we mainly focus on the case when $A$ satisfies $(*)$, though the case where $A$ does not satisfy $(*)$ does feature in our first result and also Section~\ref{secrem}. 

We use the abbreviation w.h.p. for \emph{with high probability} (with probability tending to $1$ as $n$ tends to infinity). Our first result gives the threshold for the random $(A,b)$-game whenever $Ax=b$ corresponds to a single linear equation.
\begin{thm}\label{singlemain}
Let $A$ be a fixed integer-valued matrix of dimension $1 \times k$ and $b$ a fixed integer (i.e. $Ax=b$ corresponds to a single linear equation $a_1 x_1 + \dots + a_k x_k=b$ with the $a_i$ non-zero integers). 
\begin{itemize}
\item[(i)]{If the pair $(A,b)$ is irredundant and $A$ is irredundant and satisfies $(*)$, then the $(A,b)$-game on $[n]_p$ has a threshold probability of $\Theta(n^{-\frac{k-2}{k-1}})$;}
\item[(ii)]{If the pair $(A,b)$ is irredundant and $A$ is irredundant and does not satisfy $(*)$, then the $(A,b)$-game on $[n]_p$ is Maker's win if $p \gg n^{-1/3}$ and Breaker's win if $p \ll n^{-1/3}$;}
\item[(iii)]{If the pair $(A,b)$ is irredundant and $A$ is not irredundant, then 
\begin{itemize}
\item[(a)]{the $(A,b)$-game on $[n]_p$ is Breaker's win w.h.p. for any $p=o(1)$ if the coefficients $a_i$ are all positive or all negative;}
\item[(b)]{the $(A,b)$-game on $[n]_p$ is Maker's win if $p \gg n^{-1/3}$ and Breaker's win if $p \ll n^{-1/3}$ otherwise;}
\end{itemize}}
\item[(iv)]{If the pair $(A,b)$ is not irredundant, then the $(A,b)$-game on $[n]$ is (trivially) Breaker's win.}
\end{itemize} 
\end{thm}

Note that most `interesting' equations lies in the class of equations given by (i). Here $A$ satisfying $(*)$ implies that $k \geq 3$; in particular the class given by (i) includes several natural equations that have been extensively studied, e.g. $x+y=z$, $x+y=z+t$ and $x+y=2z$. In the $(A,b)$-games corresponding to these equations, Breaker's aim is to restrict Maker's set to being a \emph{sum-free set}, a \emph{Sidon set} and a \emph{progression-free set} respectively. The remaining classes of equations given by (ii)--(iv) are all in some sense `trivial'; the proofs of these statements appear in Section~\ref{secone}.

In fact Theorem~\ref{singlemain}(i) will follow immediately from a much more general theorem. First we need some more definitions. We say that an $\ell \times k$ matrix $A$ of full rank $\ell$ is \emph{strictly balanced} if, for every $W \subseteq [k]$ for which $2 \leq |W| < k$, the inequality 
\begin{align*}
\frac{|W|-1}{|W|-1+\rank (A_{\overline{W}})-\ell} < \frac{k-1}{k-1-\ell}
\end{align*}
holds. In particular note that if $A$ is strictly balanced then $m(A)=\frac{k-1}{k-1-\ell}$ (though the converse is not true). Given an irredundant matrix $A$ which satisfies $(*)$, we define the \emph{associated matrix $B(A)$} to be a strictly balanced, irredundant matrix of full rank which satisfies $(*)$, which is found by using elementary row operations on $A$ then deleting some rows and columns, and satisfies $m(B(A))=m(A)$. The fact that such a matrix exists is not entirely obvious, and is essentially proven in~\cite{RR4}. We provide further details in Section~\ref{secbreaker}. Also note that if $A$ itself is strictly balanced then we simply have $B(A)=A$.

Let $\mu(n,A,b)$ denote the size of the largest subset of $[n]$ which does not contain a $k$-distinct solution to $Ax=b$. The main result of this paper is the following.
\begin{thm}\label{main}
Let $A$ be a fixed integer-valued matrix of dimension $\ell' \times k'$ and $b$ a fixed integer-valued vector of dimension $\ell'$. Given the pair $(A,b)$ is irredundant and $A$ is irredundant and satisfies $(*)$ we have the following:
\begin{itemize}
\item[(i)]{Let $\eps>0$. There exists a positive constant $C$ such that if $p > Cn^{-1/m(A)}$, then for any $R \subseteq [n]_p$ satisfying $|R| \leq (1-\frac{\mu(n,A,b)}{n}-\eps)np$, we have
$$ \lim_{n \to \infty} \prob \left( \text{Maker wins the $(A,b)$-game on $[n]_p \setminus R$} \right)=1.$$}
\item[(ii)]{Suppose the associated matrix $B(A)$ is an $\ell \times k$ matrix of full rank $\ell$, where $\ell$ divides $k-1$. There exists a positive constant $c$ such that if $p < cn^{-1/m(A)}$ then 
$$ \lim_{n \to \infty} \prob \left( \text{Breaker wins the $(A,b)$-game on $[n]_p$} \right)=1.$$}
\end{itemize}
\end{thm}

First note that it follows from a supersaturation result (Lemma 4.1 in~\cite{KRSS}) that for all pairs $(A,b)$ as stated in Theorem~\ref{main}, there exist $n_0=n_0(A,b)$, $\delta=\delta(A,b)>0$, such that for all integers $n \geq n_0$ we have $\mu(n,A,b) \leq (1-\delta)n$. Thus in particular Theorem~\ref{main}(i) implies that there exists a positive constant $C$ such that if $p > Cn^{-1/m(A)}$, then Maker wins the $(A,b)$-game on $[n]_p$ w.h.p. Also, note that if $A$ is a $1 \times k$ matrix with non-zero entries, then it is strictly balanced, and so $B(A)=A$. Thus $A$ is a matrix which satisfies the hypothesis of Theorem~\ref{main}(ii). Theorem~\ref{singlemain}(i) follows immediately from these two comments.

Another example of a class of pairs $(A,b)$ for which Theorem~\ref{main} gives the threshold probability up to a constant factor are all irredundant pairs $(A,b)$ such that $A$ is irredundant, has no columns consisting entirely of zeroes, satisfies $(*)$ and is of dimension $2 \times k'$ for some odd $k'$. This follows since by construction either $B(A)=A$ or $B(A)$ is a $1 \times k$ matrix for some $k < k'$. Either way, $B(A)$ then satisfies the hypothesis of Theorem~\ref{main}(ii).

For the Maker's win statement, the fact that we can delete a certain fraction of elements from $[n]_p$ and still have Maker's win w.h.p. means we have a \emph{resilience} theorem. The study of how strongly a graph or set satisfies a certain property has a rich history. An early famous example is Tur\'an's theorem, which tells us that given $r\geq 3$ we must delete a $\frac{1}{r-1}$-fraction of the edges of the complete graph in order to obtain a graph which does not contain $K_r$ as a subgraph. The \emph{global resilience} of a graph or set property generally asks how many edges or elements must be deleted in order to rid the graph or set of the property. Note that in our resilience result, the property is the game being Maker's win w.h.p., and the resilience is best possible in terms of the bound on the size of the set $R$: Indeed, since the largest subset of $[n]$ with no $k$-distinct solutions to $Ax=b$ has size $\mu(n,A,b)$, w.h.p. $[n]_p$ contains a subset $S$ of size $p(\mu(n,A,b)-\eps n)$ with no $k$-distinct solutions to $Ax=b$. Thus we can remove $(1-\frac{\mu(n,A,b)}{n}+\eps)np$ elements from $[n]_p$ to obtain $S$ (noting that a game on $S$ is trivially Breaker's win). 

It is very interesting to note the parallels between our theorem and the random Rado theorem and the resilience theorem stated below. In particular, the parameter $m(A)$ is also crucial here. First, call a set of integers $X$ \emph{$(A,b,r)$-Rado} if for any $r$-colouring of $X$, there is always a monochromatic $k$-distinct solution to $Ax=b$ in $X$. (Note $X$ being $(A,b,1)$-Rado just means there is a $k$-distinct solution to $Ax=b$ in $X$.) 

\begin{thm}[\cite{FRS,RR4}]\label{randomrado}
Let $A$ be a fixed integer-valued matrix of dimension $\ell \times k$ and let $r \geq 2$ be a positive integer. Given $A$ is irredundant and partition regular, there exist constants $C,c > 0$ such that
$$\lim_{n \to \infty} \mathbb{P} [ [n]_p \text{ is $(A,0,r)$-Rado}] = \begin{cases} 1 \text{ if $p> Cn^{-1/m(A)}$;} \\ 0 \text{ if $p<cn^{-1/m(A)}$.} \end{cases}$$
\end{thm}

\begin{thm}[\cite{HST,Spie}]\label{resilientrado}
Let $A$ be a fixed integer-valued matrix of dimension $\ell \times k$. Given $A$ is irredundant and satisfies $(*)$, for all $\eps>0$, there exists a positive constant $C$ such that if $p > Cn^{-1/m(A)}$, then for any $R \subseteq [n]_p$ satisfying $|R| \leq (1-\frac{\mu(n,A,0)}{n}-\eps)np$, we have
$$\lim_{n \to \infty} \mathbb{P} [ [n]_p \setminus R \text{ is $(A,0,1)$-Rado}] = 1.$$
\end{thm}

Note that Theorem~\ref{resilientrado} was already proven for \emph{density regular} matrices and the matrix $\begin{pmatrix} 1 & 1 & -1 \end{pmatrix}$ by Schacht~\cite{Sch}. Further, although not explicity stated in the paper, Schacht's method extends to the class of matrices stated in Theorem~\ref{resilientrado}. 

Theorem~\ref{randomrado} implies that by again using strategy-stealing, we could obtain a proof for the non-resilient version of Theorem~\ref{main}(i) for irredundant partition regular matrices. However our method as already noted achieves the best resilience possible, and further it extends to all irredundant matrices which satisfy $(*)$ (even those for which there exists a $2$-colouring of $\mathbb{N}$ with no monochromatic $k$-distinct solutions to $Ax=b$). Our proof also gives an explicit strategy. 

The proof of Theorem~\ref{main}(i) closely follows the method of Theorem 16 in~\cite{NSS}. Here, Nenadov, Steger and Stojakovi\'c consider a similar problem: the \emph{$H$-game} is where the board is the edges of a complete graph, and the winning sets are sets of edges which correspond to a copy of a fixed subgraph $H$. This game and its related Ramsey problems resemble the $(A,b)$-game as follows: Set $d_2(H):=0$ if $e(H)=0$, $d_2(H):=1/2$ if $e(H)=1$, and $d_2(H):=(e(H)-1)/(v(H)-2)$ otherwise. Then define the \emph{$2$-density} of $H$ to be $m_2(H):=\max_{H' \subseteq H} d_2(H')$. For most graphs $H$, the graph analogues of Theorems~\ref{randomrado} and~\ref{resilientrado} (the random Ramsey theorem and resilient subgraphs theorem, see~\cite{RR1,RR2,RR3,CG,Sch}) have a threshold of $\Theta(n^{-1/m_2(H)})$. Bednarska and \L uczak~\cite{BL} showed that the threshold bias for the $H$-game is $\Theta(n^{1/m_2(H)})$. Thus both the $H$-games and $(A,b)$-games (in most cases) have a threshold bias which is the inverse of the threshold for the random (respective) Ramsey/Rado theorem and the resilience theorems. Kusch, Ru\'e, Spiegel and Szab\'o~\cite{KRSS} in fact show that there is an intimate link between resilience and the threshold bias, which explains the parameters of $m(A)$ and $m_2(H)$ appearing for both. They refer to this phenomenon as the \emph{probabilistic Tur\'an intuition} for biased Maker-Breaker games; see Section 6.4 of~\cite{KRSS} for more details. 

An analogous definition of strictly balanced exists for graphs. In~\cite{NSS}, Nenadov, Steger and Stojakovi\'c show that the threshold probability for the random $H$-game is $\Theta(n^{-1/m_2(H)})$ when $H$ is strictly balanced (Theorem 2 in~\cite{NSS}). However there are a class of graphs which have a threshold probability different to that of the random Ramsey/resilient subgraph theorem and the inverse of the threshold bias (Theorem 4 in~\cite{NSS}). Indeed, this is one of the main motivations for studying the random $(A,b)$-game: For the proof of Theorem~\ref{main}(ii), we build upon the method used by R\"odl and Ruci\'nski~\cite{RR4} to prove the $0$-statement of Theorem~\ref{randomrado}. Although our Breaker win statement is `incomplete', its proof does seem to indicate that the threshold probability for the random $(A,b)$-game (for any pair $(A,b)$ which is irredundant and $A$ irredundant and satisfying $(*)$) should be the same as the random Rado threshold. That is, we hope that there is no need for the assumption that $\ell$ divides $k-1$ in Theorem~\ref{main}(ii). Also note that if we could prove our Breaker win statement for all strictly balanced matrices, then the full result would follow (see Proposition~\ref{strictred} and the paragraph following it). So interestingly in this sense, the random $(A,b)$-game does not resemble the random $H$-game.

We prove the two parts of Theorem~\ref{main} in Sections~\ref{secmaker} and~\ref{secbreaker} respectively, before finishing by proving Theorem~\ref{singlemain}(ii)--(iv) along with making some further remarks in Section~\ref{secrem}.

\section{Proof of Maker's win in Theorem~\ref{main}}\label{secmaker}
First we list a few results which are required for the proof. We will use the following simplification of Theorem 4.7 in \cite{HST}, a container result, which is itself a consequence of the general container theorems of Balogh, Morris and Samotij~\cite{BMS} and Saxton and Thomason~\cite{ST}.
\begin{thm}[\cite{HST}]\label{matcontainer}
Let $0 < \delta < 1$. Let $A$ be a fixed integer-valued matrix of dimension $\ell \times k$ and $b$ a fixed integer-valued vector of dimension $\ell$. Suppose the pair $(A,b)$ is irredundant and $A$ is irredundant and satisfies $(*)$. Let $\mathcal I(n,A,b)$ denote all sets from $\mathcal P([n])$ which contain no $k$-distinct solutions to $Ax=b$. Then there exists $D>0$ such that the following holds. For all $n \in \mathbb{N}$, there is a collection $\mathcal S \subseteq \mathcal P([n])$ and a function $f:\mathcal S \rightarrow \mathcal P([n]) $ such that:

\begin{itemize}
\item[(i)]{For all $I \in \mathcal I(n,A,b)$, there exists $S \in \mathcal S$ such that $S \subseteq I \subseteq f(S)$.}
\end{itemize}
Additionally, every $S \in \mathcal S$ satisfies
\begin{itemize}
\item[(ii)] $|S| \leq Dn^{\frac{m(A)-1}{m(A)}}$;
\item[(iii)] $S \in \mathcal I(n,A,b)$;
\item[(iv)] $f(S)$ contains at most $\delta n^{k-\ell}$ $k$-distinct solutions to $Ax=b$; and
\item[(v)] $|f(S)| \leq \mu(n,A,b)+\delta n$.
\end{itemize}
\end{thm}

Note that in~\cite{HST} the theorem is stated for the homogeneous case $Ax=0$ only. However, the result easily generalises to the inhomogeneous case $Ax=b$. Full details can be found in~\cite{Hanc}. 

An upper bound on the size of the largest subset of $[n]$ containing no $k$-distinct solutions to $Ax=b$ is also required. The following is a consequence of Lemma 4.1 in~\cite{KRSS} and Theorem 2 in~\cite{KSV}.

\begin{thm}\label{mubound}
Let $A$ be a fixed integer-valued matrix of dimension $\ell \times k$ and $b$ a fixed integer-valued vector of dimension $\ell$. Given the pair $(A,b)$ is irredundant and $A$ is irredundant and satisfies $(*)$, then there exist $n_0 \in \mathbb{N}$ and $\delta>0$ such that for all integers $n \geq n_0$ we have $\mu(n,A,b) \leq (1-\delta)n$.
\end{thm}

We will need the Markov and Chernoff inequalities.

\begin{prop}\label{Markov} Let X be a non-negative random variable. Then for all $t>0$ we have $\mathbb P [X \geq t] \leq \frac{\mathbb E [X]}{t}$.
\end{prop}

\begin{prop}\label{Chernoff} Let $X_1, \dots, X_n$ be independent Bernoulli distributed random variables with $\mathbb P [X_i=1]=p$ and $\mathbb P [X_i=0]=1-p$. Then for $X= \sum_{i=1}^n X_i$ and every $0 < \delta \leq 1$, we have $$\mathbb P [X \leq (1-\delta) \mathbb E [X]] \leq e^{-\mathbb E [X] \delta^2 /2}.$$
\end{prop}

Finally, we require the Erd\H{o}s-Selfridge Criterion, commonly used to prove a Breaker strategy result, which we mentioned in the introduction. (Note that we do mean Breaker here; in our proof, we create an auxiliary game where the original Maker needs to play the role of Breaker!)

\begin{thm}[\cite{ES}]\label{ESCri} Let $X$ be a set and let $\mathcal{F}$ be a family of subsets of $X$. Then if Breaker has the first move in the game, and $$ \sum_{A \in \mathcal{F}} 2^{-|A|} <1,$$ then Breaker has a winning strategy for the Maker-Breaker game $(X,\mathcal{F})$.
\end{thm}

\begin{proof}[Proof of Theorem~\ref{main}(i)]
Apply Theorem~\ref{mubound} with parameters $A,b$ to obtain $\eps'>0$ such that $\mu(n,A,b) \leq (1-\eps')n$ for sufficiently large $n$. Let $\eps>0$ noting that without loss of generality we can assume $\eps \ll \eps'$. Suppose $n$ is sufficiently large. Apply Theorem~\ref{matcontainer} with parameters $\eps/4,A,b$ to obtain $D>0$, a collection $\mathcal{S} \subseteq \mathcal{P}([n])$ and a function $f$ satisfying Theorem~\ref{matcontainer}(i)--(v). Fix $\delta \ll \eps$ and choose $C$ such that $0 \leq 1/C \ll 1/D,\delta,\eps$. Let $p>Cn^{-1/m(A)}$. Note that $m(A)>1$ (see Proposition 4.3 in~\cite{HST}) and thus $pn$ tends to infinity as $n$ tends to infinity. Let $R$ be as in the statement of the theorem and set $X:=[n]_p \setminus R$.

Maker's aim is to claim a $k$-distinct solution to $Ax=b$ within $X$, and Breaker's aim is to prevent this. If Maker loses, then her set $M \subseteq X$ does not contain a $k$-distinct solution to $Ax=b$. Hence $M \in \mathcal{I}(n,A,b)$ and so there exists $S \in \mathcal{S}$ such that $S \subseteq M \subseteq f(S)$ and $S \subseteq X$. Given $S \in \mathcal{S}$ note that if Maker claims one element from $X \setminus f(S)$ then $M \not \subseteq f(S)$; hence consider the auxiliary game $(X,\mathcal{F})$ where $$\mathcal{F}:=\{ X \setminus f(S) : S \in \mathcal{S} \text{ and } S \subseteq X\}.$$ Maker can ensure that she wins the $(A,b)$-game on $X$ by picking at least one element from each set in $\mathcal{F}$, that is, she wins the auxiliary game as Breaker. We now make the following claim about the auxiliary game. 

\begin{claim}\label{makertac}
\begin{itemize}
\item[(i)]{For all $S \in \mathcal{S}$ such that $S \subseteq X$, we have $|X \setminus f(S)| \geq \eps np/2$ w.h.p.}
\item[(ii)]{We have $|\mathcal{F}| \leq 2^{\eps np/4}$ w.h.p.}
\end{itemize}
\end{claim}

Assuming the claim holds, it now easily follows that 
$$ \sum_{A \in \mathcal{F}} 2^{-|A|} \leq  2^{\eps np/4} \cdot 2^{-\eps np/2} < 1,$$
that is, the hypothesis of Theorem~\ref{ESCri} holds for the game $(X,\mathcal{F})$. Thus Maker wins the game as Breaker in the auxiliary game, and thus wins the original game (the $(A,b)$-game on $X$). Since this happens w.h.p., it remains to prove the claim.

\begin{proof}[Proof of Claim~\ref{makertac}]
First we shall count $|\mathcal{F}|$. We wish to count the number of $S \in \mathcal{S}$ such that $S \subseteq X$. Recall that every $S \in \mathcal{S}$ satisfies $|S| \leq Dn^{1-1/m(A)} \leq Dpn/C$ and there are at most $\binom{n}{s}$ sets $S \in \mathcal{S}$ of size $s$. Thus we have
\begin{align}\label{makerc2}
\mathbb{E} [|\mathcal{F}|] \leq & \sum_{S \in \mathcal{S}} \mathbb P [S \subseteq [n]_p] \leq \sum_{S \in \mathcal{S}} p^{|S|} \leq \sum_{s=0}^{Dpn/C} \binom{n}{s} p^s \leq  (Dpn/C+1) \binom{n}{Dpn/C} p^{Dpn/C} \\  \nonumber
\leq & \, (Dpn/C+1) \left ( \frac{Ce}{D} \right )^{Dpn/C} \leq e^{\delta np} \leq 2^{\eps np/8},
\end{align}
where the last two inequalities follows by our choice of $C$ and since $\delta \ll \eps$ respectively. Thus by Proposition~\ref{Markov} we have $$\mathbb P [|\mathcal{F}| \geq 2^{\eps np/4}] \leq 2^{-\eps np/8},$$ which tends to zero as $n$ tends to infinity, proving (ii).

Now for $(i)$, observe that if we show that the probability that there exists $S \in \mathcal{S}$ such that $S \subseteq X$ and $|X \setminus f(S)| \leq \eps np/2$ tends to zero as $n$ tends to infinity, we will be done. First observe by Theorem~\ref{matcontainer} that for all $S \in \mathcal{S}$ we have $|f(S)| \leq \mu(n,A,b)+\eps n/4$ and so $|[n] \setminus f(S)| \geq n-\mu(n,A,b)-\eps n/4$. Let $\gamma:=\eps/(4-4\mu(n,A,b)/n-\eps)$ and $Y:=[n]_p \setminus f(S)$ (noting $\gamma>0$ since $\eps \ll \eps'$). By Proposition~\ref{Chernoff} we have  
\begin{align*}
& \, \mathbb P \left [ |([n] \setminus f(S)) \cap [n]_p| < \left (1-\frac{\mu(n,A,b)}{n}-\frac{\eps}{2} \right) np \right] \leq \mathbb P \left [ |Y| < (1- \gamma) \mathbb E [ |Y| ]  \right] \\
\leq & \, e^{-\mathbb E [|Y|] \gamma^2/2} \leq e^{-2 \delta np}, 
\end{align*}
where the last inequality follows since $\delta \ll \eps \ll \eps'$. Note that since $|R| \leq ( 1-\frac{\mu(n,A,b)}{n}-\eps)np$ and $X \setminus f(S) = Y \setminus R$ we have 
\begin{align}\label{makerc1}
\mathbb P [ |X \setminus f(S)| < \eps np/2] \leq e^{-2 \delta np},
\end{align}
for all $S \in \mathcal{S}$. Also since $S \subseteq f(S)$, the events $S \subseteq [n]_p$ and $|X \setminus f(S)|$ being small are independent. Thus
\begin{align*}
& \mathbb P [ \text{There exists $S \in \mathcal{S}$ such that $S \subseteq [n]_p$ and $|X \setminus f(S)|< \eps np/2$}] \\ 
\leq & \sum_{S \in \mathcal{S}} \mathbb P [S \subseteq [n]_p \text{ and } |X \setminus f(S)|< \eps np/2] \\
\leq & \sum_{S \in \mathcal{S}} \left( \mathbb P [S \subseteq [n]_p] \cdot \mathbb P[|X \setminus f(S)|< \eps np/2] \right) \stackrel{(\ref{makerc1})}{\leq} e^{-2 \delta np} \sum_{S \in \mathcal{S}} \mathbb P [S \subseteq [n]_p] \\
\stackrel{(\ref{makerc2})}{\leq} & \, e^{-2 \delta np} \cdot e^{\delta np} = e^{- \delta np},
\end{align*}
which tends to zero as $n$ tends to infinity, as required.
\end{proof}
\end{proof}

\section{Proof of Breaker's win in Theorem~\ref{main}}\label{secbreaker}

The proof will follow a similar tactic to that used by R\"odl and Ruci\'nski~\cite{RR4} for their proof of the $0$-statement of Theorem~\ref{randomrado}. Recall that the goal of R\"odl and Ruci\'nski was to show that, given an irredundant partition regular matrix $A$, an integer $r \geq 2$, and an upper bound on the probability $p$, then w.h.p. there exists an $r$-colouring of $[n]_p$ such that there are no monochromatic $k$-distinct solutions to $Ax=0$. The proof consisted of three parts:

\begin{enumerate}[label=(P\arabic*)]
\item{{\bf A reduction of the problem.} It is shown that it suffices to prove the result for the associated matrix $B(A)$. The problem is then rephrased to one about an \emph{associated hypergraph}.}\label{P1}
\item{{\bf A deterministic lemma.} It is shown that if all $r$-colourings of $[n]_p$ contain a monochromatic $k$-distinct solution to $Ax=0$, then the associated hypergraph must contain a certain connected subhypergraph.}\label{P2}
\item{{\bf A probabilistic lemma.} It is shown that if $p < cn^{-1/m(A)}$, then w.h.p. the associated hypergraph does not contain the subhypergraph given by the deterministic lemma.}\label{P3}
\end{enumerate}

Recall that our aim is to show that under the hypothesis of Theorem~\ref{main}(ii), 
w.h.p. Breaker wins the $(A,b)$-game on $[n]_p$. Our proof consists of the same three general parts, with appropriate amendments to the lemmas. 
\begin{enumerate}[label=(Q\arabic*)]
\item{{\bf A reduction of the problem.} As \ref{P1} above.}\label{Q1}
\item{{\bf Two deterministic lemmas.} These together show that if Maker wins the $(A,b)$-game on $[n]_p$, then the associated hypergraph must contain a certain connected subhypergraph.}\label{Q2}
\item{{\bf A probabilistic lemma.} It is shown that if $B(A)$ is an $\ell \times k$ matrix of full rank $\ell$ which satisfies $\ell$ divides $k-1$, and $p < cn^{-1/m(A)}$, then w.h.p. 
the associated hypergraph does not contain the subhypergraph given by the deterministic lemmas.}\label{Q3}
\end{enumerate}
We will of course make this more rigorous as we get to each part of the proof. \\ 

{\bf \ref{Q1} A reduction of the problem.} 
First we show that in order for Breaker to win the $(A,b)$-game on any set of integers $X$, its suffices to show that Breaker wins the $(B,b')$-game on $X$, for some matrix $B:=B(A)$ and vector $b':=b'(A,b)$. For a vector $x=(x_1,\dots,x_k)$ and a non-empty set $W \subseteq [k]$, let $x_W:=(x_i)_{i \in W}$.

\begin{prop}[\cite{KRSS}, Corollary 4.3 and Lemma 4.2]\label{strictred}
Let $A$ be a fixed integer-valued matrix of dimension $\ell \times k$ and $b$ a fixed integer-valued vector of dimension $\ell$. Suppose the pair $(A,b)$ is irredundant and $A$ is irredundant and satisfies $(*)$. There exists a non-empty set $W \subseteq [k]$, 
a matrix $B$ of full rank which is irredundant, satisfies $(*)$, and is strictly balanced,
and a vector $b'$ for which the pair $(B,b')$ is irredundant,
such that if $Ax=b$, then $Bx_W=b'$.
\end{prop}

Note that the homogeneous case for where $b,b'$ are zero vectors is implicitly stated in~\cite{RR4}. We call the pair $(B,b')$ above the \emph{associated pair of $(A,b)$}, and call $B=B(A)$ the \emph{associated matrix of $A$}. The consequence for us of Proposition~\ref{strictred} is that if Maker wins the $(A,b)$-game, then Maker also wins the $(B,b')$-game (since a solution to $Ax=b$ always gives rise to a solution to $Bx'=b'$). Thus in order to prove Breaker wins the $(A,b)$-game, it suffices to prove that Breaker wins the $(B,b')$-game.

With any game $(X,\mathcal{F})$ there exists an associated hypergraph $H(X,\mathcal{F})$ with vertex set $X$ and edge set $\mathcal{F}$. Write $H(X,B,b'):=H(X,\mathcal{F})$ to represent the hypergraph where $X$ is a set of integers, and $\mathcal{F}$ is the set of all $k$-distinct solutions to $Bx'=b'$ (assuming that $B$ is an $\ell \times k$ matrix). Thus we may think of the game as one where Maker and Breaker take turns claiming a vertex of the $k$-uniform hypergraph $H([n]_p,B,b')$ and Maker's aim is to obtain an edge of $H([n]_p,B,b')$, and Breaker's aim is to prevent this. For the remainder of the proof we will assume that we have fixed $A$ and $b$ (and therefore $B$ and $b'$), and set $H:=H([n]_p,B,b')$. We will assume that $B$ is an $\ell \times k$ matrix, and note that since $B$ satisfies property $(*)$, it is easy to see that we must have $k \geq 3$.

{\bf Hypergraph notation.}
We now introduce some notation which will be required for the deterministic and probabilistic lemmas. For a $k$-uniform hypergraph $H$ with edge set $E:=E(H)$ and vertex set $V:=V(H)$, let an \emph{edge order} be an enumeration of the edges $E$.  
For a given edge order of $E$ and edge $e \in E$, call a vertex $v \in e$ \emph{new in $e$} if $v$ did not appear in any edge which came before $e$ in the edge order. Otherwise call $v$ \emph{old in $e$}. 
We call an edge $e$ \emph{good} if it has precisely one old vertex, \emph{bad} if it has between two and $k-1$ old vertices, and \emph{$k$-bad} if it has $k$ old vertices. Note that we always consider edges to be good, bad, or $k$-bad with respect to a given edge order; similarly whether a vertex is new or old in a given edge also depends on the given edge order. So throughout we will make it clear which edge order we are referring to. Note that given an edge order, a vertex will always be new in precisely one edge (and old in every other edge it appears in). For ease of notation we may sometimes identify a hypergraph with an edge order of its edges, e.g. if we have $P:=e_0,\dots,e_t$, then we consider the hypergraph $P$ to have $E(P):=\{e_0,\dots,e_t\}$ and $V(P):=\{x \in e, e \in E(P)\}$. 

Let $e_0,\dots,e_t$ be an edge order. We call the edge order \emph{allowed} if for all $i \in [t]$, $e_i$ is good, bad or $k$-bad (that is, there is no edge $e_i$ with $i \geq 1$ such that $e_i$ is vertex-disjoint from all the edges $e_0, \dots, e_{i-1}$). We call it \emph{valid} if for all $i \in [t]$, $e_i$ is good or bad. It is \emph{simple} if for all $i \in [t]$, $e_i$ is good. 
For a subset of edges $e_{f_1},\dots,e_{f_u}$ of $e_0,\dots,e_t$ we do not assume $f_1 \leq \dots \leq f_u$ unless otherwise stated. 
For two vertex-disjoint sets $X_1, X_2 \subseteq V(H)$, we define a \emph{minimal path from $X_1$ to $X_2$ in $\{e_0,\dots,e_t\}$} to be a subset of edges $e_{f_1}, \dots, e_{f_u}$ of $e_0,\dots,e_t$ such that
\begin{itemize}
\item[(i)]{we have $X_1 \cap e_{f_1} \not = \emptyset$, and $x \notin e_{f_a}$ for any $a \geq 2$ and $x \in X_1$;}
\item[(ii)]{we have $X_2 \cap e_{f_u} \not = \emptyset$, and $x \notin e_{f_a}$ for any $a \leq u-1$ and $x \in X_2$;}
\item[(iii)]{for all $i,j \in [u]$ with $i<j$ we have $|e_{f_i} \cap e_{f_j}| \geq 1$ if $i=j-1$ and $|e_{f_i} \cap e_{f_j}|=0$ otherwise.}
\end{itemize}
As a small example consider the second hypergraph in Figure 1 (with edges labelled):
\begin{itemize}
\item{The edge order $e_1,e_2,\dots,e_8$ is valid, but not simple, since all $e_i$ for $2 \leq i \leq 8$ are good or bad, and in particular $e_2$ is bad.}
\item{The edge order $e_1,e_4,e_3$ is not allowed, since $e_4$ is vertex-disjoint from $e_1$.}
\item{Setting $X_1:=e_1$ and $X_2:=e_5$, we see that $e_3,e_4$ is a minimal path from $X_1$ to $X_2$ in $\{e_2,e_3,e_4\}$, whereas $e_2,e_3,e_4$ is not; condition (i) of a minimal path is violated since there exists a vertex $x \in X_1 \cap e_3$.}
\end{itemize} 
We now give names to a variety of $k$-uniform hypergraphs which will appear in our deterministic and probabilistic lemmas. Suppose that $e_{f_1},\dots,e_{f_u}$ for some $u \in \mathbb{N}$ is a valid edge order, where $|e_{f_i} \cap e_{f_{i+1}}| \geq 1$ for all $i \in [u-1]$. We call $e_{f_1},\dots,e_{f_u}$:
\begin{itemize}
\item{An \emph{overlapping pair}, if $u=2$ and $2 \leq |e_{f_1} \cap e_{f_2}| \leq k-1$;}
\item{A \emph{loose cycle}, if $u \geq 3$, and for all $i,j \in [u]$ with $i<j$ we have
$$|e_{f_i} \cap e_{f_j}| = \begin{cases} 1 \text{ if $i = j-1$, or $i=1$ and $j=u$;} \\ 0 \text { otherwise;} \end{cases}$$}
\item{A \emph{loose path}, if for all $i,j \in [u]$ with $i<j$ we have
$$|e_{f_i} \cap e_{f_j}| = \begin{cases} 1 \text{ if $i = j-1$;} \\ 0 \text { otherwise;} \end{cases}$$   }
\item{A \emph{spoiled cycle}, if $P_1:=e_{f_1},e_{f_2}$ forms an overlapping pair, $P_2:=e_{f_3},\dots,e_{f_u}$ forms a loose path, and $P_1$ and $P_2$ are vertex-disjoint except for two vertices $x \not =y$, where we have $x=(e_{f_2} \setminus e_{f_1} ) \cap (e_{f_3} \setminus e_{f_z})$ (where $z=4$ if $u \geq 4$, and $z=1$ otherwise) and $y=(e_{f_1} \setminus e_{f_2}) \cap (e_{f_u} \setminus e_{f_{u-1}})$;}
\item{A \emph{double loose cycle}, if for some $v \leq u-2$, $P_1:=e_{f_1},\dots,e_{f_v}$ forms a loose cycle, $P_2:=e_{f_{v+1}},\dots,e_{f_u}$ forms a loose path, and $P_1$ and $P_2$ are vertex-disjoint except for two vertices $x \not =y$, where we have $x=(e_{f_{v+1}} \setminus e_{f_{v+2}}) \cap e_{f_v}$ and $y=(e_{f_u} \setminus e_{f_{u-1}}) \cap e_{f_a}$ for some $a \in [v]$;}
\item{A \emph{double overlapping pair}, if $u=4$, $e_{f_1},e_{f_2}$ and $e_{f_3},e_{f_4}$ each form overlapping pairs, which are vertex-disjoint except for two vertices $x \not =y$, where we have $x=(e_{f_1} \setminus e_{f_2}) \cap (e_{f_4} \setminus e_{f_3})$ and $y=(e_{f_2} \setminus e_{f_1}) \cap (e_{f_3} \setminus e_{f_4})$, and $|e_{f_3} \cap e_{f_4}| \leq k-2$;}
\item{An \emph{overlapping pair/loose cycle with handle}, if $e_{f_1},\dots,e_{f_{u-1}}$ forms an overlapping pair/loose cycle and $e_{f_u}$ is bad in the edge order $e_{f_1},\dots,e_{f_u}$;}
\item{An \emph{overlapping pair/loose cycle to overlapping pair/loose cycle}, if for some $w \leq v<u$, $P_1:=e_{f_1},\dots,e_{f_w}$ forms an overlapping pair or loose cycle, $P_2:=e_{f_{w+1}},\dots,e_{f_v}$ forms a loose path and $P_3:=e_{f_{v+1}},\dots,e_{f_u}$ forms an overlapping pair or loose cycle; moreover if $w=v$ then $|V(P_1) \cap V(P_3)|=1$; otherwise $|V(P_1) \cap V(P_2)|=1$, $V(P_1) \cap V(P_3) = \emptyset$, $|V(P_2) \cap V(P_3)|=1$, and additionally if $w\leq v-2$, then $e_{f_{w+2}} \cap V(P_1) = \emptyset$ and 
$e_{f_{v-1}} \cap V(P_3) = \emptyset$.}
\end{itemize}

\begin{figure}[h]
\begin{tikzpicture}
\node[vertex] at (0,2) (1) {};
\node[vertex] at (1,2) (2) {};
\node[vertex] at (2,1.333) (3) {};
\node[vertex] at (2,2.667) (4) {};
\node[vertex] at (3,0.667) (5) {};
\node[vertex] at (3,3.333) (6) {};
\node[vertex] at (4,0) (7) {};
\node[vertex] at (4,4) (8) {};
\node[vertex] at (4,1) (9) {};
\node[vertex] at (4,2) (10) {};
\node[vertex] at (4,3) (11) {};

\filldraw[fill opacity=0,fill=white!70] 
($(1)+(0,0.4)$) to ($(2)+(0,0.4)$) to ($(8) + (-0.2,0.3464)$) to[out=30,in=120] ($(8)+(0.3464,0.2)$) to[out=300,in=30] ($(8)+(0.2,-0.3464)$) to ($(2)+(0.1732,-0.4)$) to ($(1)+(0,-0.4)$) to[out=180,in=270] ($(1)+(-0.4,0)$) to[out=90,in=180] ($(1)+(0,0.4)$); 

\filldraw[fill opacity=0,fill=white!70] 
($(1)+(0,-0.5)$) to ($(2)+(0,-0.5)$) to ($(7) + (-0.25,-0.433)$) to[out=330,in=240] ($(7)+(0.433,-0.25)$) to[out=60,in=330] ($(7)+(0.25,0.433)$) to ($(2)+(0.2065,0.5)$) to ($(1)+(0,0.5)$) to[out=180,in=90] ($(1)+(-0.5,0)$) to[out=270,in=180] ($(1)+(0,-0.5)$); 

\filldraw[fill opacity=0,fill=white!70] 
($(8) + (0,0.3)$) to[out=0,in=90] ($(8) + (0.3,0)$) to ($(7) + (0.3,0)$) to[out=270,in=0] ($(7) + (0,-0.3)$) to[out=180,in=270] ($(7) + (-0.3,0)$) to ($(8) + (-0.3,0)$) to[out=90,in=180] ($(8) + (0,0.3)$);

\node at (2,-1.0) {An overlapping pair with handle};
\node at (2,-1.5) {Also a spoiled cycle};
\end{tikzpicture}
\qquad
\begin{tikzpicture}[scale=1.0]
\node[vertex] at (0,0) (1) {};
\node[vertex] at (0,1) (2) {};
\node[vertex] at (0,2) (3) {};
\node[vertex] at (0,3) (4) {};
\node[vertex] at (1,1) (5) {};
\node[vertex] at (2,1) (6) {};
\node[vertex] at (3,1) (7) {};
\node[vertex] at (4,1) (8) {};
\node[vertex] at (4,0) (9) {};
\node[vertex] at (4,2) (10) {};
\node[vertex] at (5,2) (11) {};
\node[vertex] at (5,0) (12) {};
\node[vertex] at (6,1) (13) {};
\node[vertex] at (6,0) (14) {};
\node[vertex] at (6,2) (15) {};
\node at (-0.7,3) {$e_1$};
\node at (-0.7,0) {$e_2$};
\node at (1.2,0.3) {$e_3$};
\node at (2.8,0.3) {$e_4$};
\node at (3.3,2) {$e_5$};
\node at (5,2.7) {$e_6$};
\node at (6.7,1) {$e_7$};
\node at (5,-0.7) {$e_8$};

\filldraw[fill opacity=0,fill=white!70] 
($(3) + (0,0.3)$) to[out=0,in=90] ($(3) + (0.3,0)$) to ($(1) + (0.3,0)$) to[out=270,in=0] ($(1) + (0,-0.3)$) to[out=180,in=270] ($(1) + (-0.3,0)$) to ($(3) + (-0.3,0)$) to[out=90,in=180] ($(3) + (0,0.3)$);
\filldraw[fill opacity=0,fill=white!70] 
($(4) + (0,0.4)$) to[out=0,in=90] ($(4) + (0.4,0)$) to ($(2) + (0.4,0)$) to[out=270,in=0] ($(2) + (0,-0.4)$) to[out=180,in=270] ($(2) + (-0.4,0)$) to ($(4) + (-0.4,0)$) to[out=90,in=180] ($(4) + (0,0.4)$);
\filldraw[fill opacity=0,fill=white!70] 
($(10) + (0,0.4)$) to[out=0,in=90] ($(10) + (0.4,0)$) to ($(9) + (0.4,0)$) to[out=270,in=0] ($(9) + (0,-0.4)$) to[out=180,in=270] ($(9) + (-0.4,0)$) to ($(10) + (-0.4,0)$) to[out=90,in=180] ($(10) + (0,0.4)$);
\filldraw[fill opacity=0,fill=white!70] 
($(15) + (0,0.4)$) to[out=0,in=90] ($(15) + (0.4,0)$) to ($(14) + (0.4,0)$) to[out=270,in=0] ($(14) + (0,-0.4)$) to[out=180,in=270] ($(14) + (-0.4,0)$) to ($(15) + (-0.4,0)$) to[out=90,in=180] ($(15) + (0,0.4)$);
\filldraw[fill opacity=0,fill=white!70] 
($(2)+(0,0.5)$) to ($(6)+(0,0.5)$) to[out=0,in=90] ($(6) + (0.5,0)$) to[out=270,in=0] ($(6) + (0,-0.5)$) to ($(2) + (0,-0.5)$) to[out=180,in=270] ($(2) + (-0.5,0)$) to[out=90,in=180] ($(2)+(0,0.5)$); 
\filldraw[fill opacity=0,fill=white!70] 
($(6)+(0,0.3)$) to ($(8)+(0,0.3)$) to[out=0,in=90] ($(8) + (0.3,0)$) to[out=270,in=0] ($(8) + (0,-0.3)$) to ($(6) + (0,-0.3)$) to[out=180,in=270] ($(6) + (-0.3,0)$) to[out=90,in=180] ($(6)+(0,0.3)$); 
\filldraw[fill opacity=0,fill=white!70] 
($(10)+(0,0.3)$) to ($(15)+(0,0.3)$) to[out=0,in=90] ($(15) + (0.3,0)$) to[out=270,in=0] ($(15) + (0,-0.3)$) to ($(10) + (0,-0.3)$) to[out=180,in=270] ($(10) + (-0.3,0)$) to[out=90,in=180] ($(10)+(0,0.3)$); 
\filldraw[fill opacity=0,fill=white!70] 
($(9)+(0,0.3)$) to ($(14)+(0,0.3)$) to[out=0,in=90] ($(14) + (0.3,0)$) to[out=270,in=0] ($(14) + (0,-0.3)$) to ($(9) + (0,-0.3)$) to[out=180,in=270] ($(9) + (-0.3,0)$) to[out=90,in=180] ($(9)+(0,0.3)$); 
\node at (3,-1.5) {An overlapping pair to loose cycle};
\end{tikzpicture}
\caption{Examples of our hypergraphs.}
\end{figure}

Note that since we identify hypergraphs with one of their allowed edge orders, a hypergraph may fit the description of more than one of the above (e.g. a hypergraph could be both a spoiled cycle and an overlapping pair with handle, see Figure $1$). We define a \emph{bicycle} to be a hypergraph which is one of:
\begin{itemize}
\item{a spoiled cycle;}
\item{a double overlapping pair/loose cycle;}
\item{an overlapping pair/loose cycle with handle;}
\item{an overlapping pair/loose cycle to overlapping pair/loose cycle.}
\end{itemize}

Suppose that $e_{f_1},\dots,e_{f_u}$ for some $u \in \mathbb{N}$ is an allowed edge order. We call $e_{f_1},\dots,e_{f_u}$:
\begin{itemize}
\item{A \emph{Pasch configuration} if $k=3$, $u=4$, there are six vertices within the four edges, and each of these appear in precisely two of the edges (one vertex for each of the six pairs of edges); see Figure $2$;}
\item{A \emph{$k$-uniform loose $u$-star}, if the edges are completely disjoint except for all intersecting in one `central vertex';}
\item{A \emph{$(k,u/2,2)$-star}, if $u$ is even and the edges form two $k$-uniform loose $(u/2)$-stars $S_1$ and $S_2$, and there is a bijection $f$ between the edges of $S_1$ to those of $S_2$ such that $e$ and $f(e)$ share all their vertices except for the two central vertices, for each edge $e$ in $S_1$ (so a $(k,u,2)$-star has one more vertex than a $k$-uniform loose $u$-star, but has twice as many edges);}
\item{A \emph{$(k,u,a)$-link}, if there are $k+a$ vertices within the $u$ edges, and given any $i,j \in [u]$ with $i<j$ we have $|e_{f_i} \cup e_{f_j}|=k+a$ (i.e. any pair of the edges contain all $k+a$ vertices between them).}
\end{itemize}  
Observe that a $(k,u,a)$-link with $u\geq 3$ must have $a \leq \floor{k/(u-1)}$.

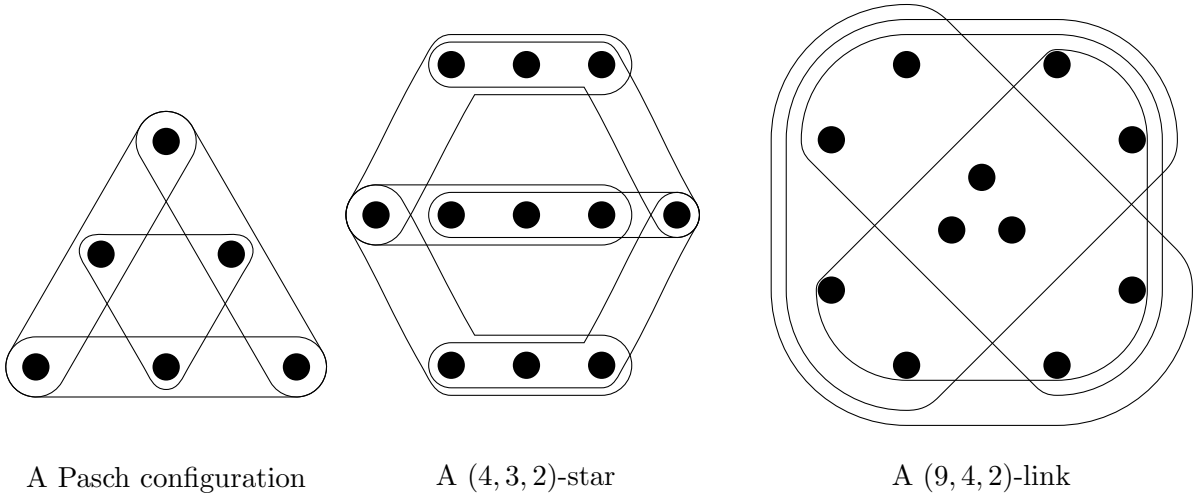
\begin{figure}[h]
\begin{tikzpicture}
\node[vertex] at (30:1) (1) {};
\node[vertex] at (90:2) (2) {};
\node[vertex] at (150:1) (3) {};
\node[vertex] at (210:2) (4) {};
\node[vertex] at (270:1) (5) {};
\node[vertex] at (330:2) (6) {};

\filldraw[fill opacity=0,fill=white!70] 
($(2)+(-0.3464,0.2)$) to ($(4)+(-0.3464,0.2)$) to[out=240,in=150] ($(4) + (-0.2,-0.3464)$) to[out=330,in=240] ($(4) + (0.3464,-0.2)$) to ($(2) + (0.3464,-0.2)$) to[out=60,in=330] ($(2) + (0.2,0.3464)$) to[out=150,in=60] ($(2)+(-0.3464,0.2)$); 

\filldraw[fill opacity=0,fill=white!70] 
($(2)+(-0.3464,-0.2)$) to ($(6)+(-0.3464,-0.2)$) to[out=300,in=210] ($(6) + (0.2,-0.3464)$) to[out=30,in=300] ($(6) + (0.3464,0.2)$) to ($(2) + (0.3464,0.2)$) to[out=120,in=30] ($(2) + (-0.2,0.3464)$) to[out=210,in=120] ($(2)+(-0.3464,-0.2)$); 

\filldraw[fill opacity=0,fill=white!70] 
($(4)+(0,0.4)$) to ($(6)+(0,0.4)$) to[out=0,in=90] ($(6) + (0.4,0)$) to[out=270,in=0] ($(6) + (0,-0.4)$) to ($(4) + (0,-0.4)$) to[out=180,in=270] ($(4) + (-0.4,0)$) to[out=90,in=180] ($(4)+(0,0.4)$); 


\filldraw[fill opacity=0,fill=white!70] 
($(3)+(-0.2598,0.15)$) to[out=60,in=180] ($(3)+(0.15,0.2598)$) to ($(1)+(-0.15,0.2598)$) to[out=0,in=120] ($(1)+(0.2598,0.15)$) to[out=300,in=60] ($(1)+(0.15,-0.2598)$) to ($(5)+(0.3,0)$) to[out=240,in=0]  ($(5)+(0,-0.3)$) to[out=180,in=300] ($(5)+(-0.3,0)$) to ($(3)+(-0.15,-0.2598)$) to[out=120,in=240] ($(3)+(-0.2598,0.15)$);

\node at (0,-2.5,0) {A Pasch configuration};
\end{tikzpicture}
\begin{tikzpicture}
\node[vertex] at (1,0) (1) {};
\node[vertex] at (2,0) (2) {};
\node[vertex] at (3,0) (3) {};
\node[vertex] at (1,2) (4) {};
\node[vertex] at (2,2) (5) {};
\node[vertex] at (3,2) (6) {};
\node[vertex] at (1,4) (7) {};
\node[vertex] at (2,4) (8) {};
\node[vertex] at (3,4) (9) {};
\node[vertex] at (0,2) (10) {};
\node[vertex] at (4,2) (11) {};

\filldraw[fill opacity=0,fill=white!70] 
($(10)+(0,0.4)$) to ($(6)+(0,0.4)$) to[out=0,in=90] ($(6) + (0.4,0)$) to[out=270,in=0] ($(6) + (0,-0.4)$) to ($(10) + (0,-0.4)$) to[out=180,in=270] ($(10) + (-0.4,0)$) to[out=90,in=180] ($(10)+(0,0.4)$); 

\filldraw[fill opacity=0,fill=white!70] 
($(7)+(0,0.4)$) to ($(9)+(0,0.4)$) to[out=0,in=90] ($(9) + (0.4,0)$) to[out=270,in=0] ($(9) + (0,-0.4)$) to ($(7) + (0.3116,-0.4)$) to[out=240,in=60] ($(10) + (0.3464,-0.2)$) to[out=240,in=330] ($(10) + (-0.2,-0.3464)$) to[out=150,in=240]  ($(10) + (-0.3464,+0.2)$) to[out=60,in=240] ($(7) + (-0.4,0)$) to[out=60,in=180] ($(7)+(0,0.4)$);

\filldraw[fill opacity=0,fill=white!70] 
($(1)+(0,-0.4)$) to ($(3)+(0,-0.4)$) to[out=0,in=270] ($(3) + (0.4,0)$) to[out=90,in=0] ($(3) + (0,0.4)$) to ($(1) + (0.3116,0.4)$) to[out=120,in=300] ($(10) + (0.3464,0.2)$) to[out=120,in=30] ($(10) + (-0.2,0.3464)$) to[out=210,in=120]  ($(10) + (-0.3464,-0.2)$) to[out=300,in=120] ($(1) + (-0.4,0)$) to[out=300,in=180] ($(1)+(0,-0.4)$);

\filldraw[fill opacity=0,fill=white!70] 
($(4)+(0,0.3)$) to ($(11)+(0,0.3)$) to[out=0,in=90] ($(11) + (0.3,0)$) to[out=270,in=0] ($(11) + (0,-0.3)$) to ($(4) + (0,-0.3)$) to[out=180,in=270] ($(4) + (-0.3,0)$) to[out=90,in=180] ($(4)+(0,0.3)$); 

\filldraw[fill opacity=0,fill=white!70] 
($(9)+(0,0.3)$) to ($(7)+(0,0.3)$) to[out=180,in=90] ($(7) + (-0.3,0)$) to[out=270,in=180] ($(7) + (0,-0.3)$) to ($(9) + (-0.2337,-0.3)$) to[out=300,in=120] ($(11) + (-0.2598,-0.15)$) to[out=300,in=210] ($(11) + (0.15,-0.2598)$) to[out=30,in=300]  ($(11) + (0.2598,+0.15)$) to[out=120,in=300] ($(9) + (0.3,0)$) to[out=120,in=0] ($(9)+(0,0.3)$);

\filldraw[fill opacity=0,fill=white!70] 
($(3)+(0,-0.3)$) to ($(1)+(0,-0.3)$) to[out=180,in=270] ($(1) + (-0.3,0)$) to[out=90,in=180] ($(1) + (0,0.3)$) to ($(3) + (-0.2337,0.3)$) to[out=60,in=240] ($(11) + (-0.2598,0.15)$) to[out=60,in=150] ($(11) + (0.15,0.2598)$) to[out=330,in=60]  ($(11) + (0.2598,-0.15)$) to[out=240,in=60] ($(3) + (0.3,0)$) to[out=240,in=0] ($(3)+(0,-0.3)$);

\node at (2,-1.5) {A $(4,3,2)$-star};
\end{tikzpicture}
\qquad
\begin{tikzpicture}[scale=1.0]
\node[vertex] at (1,0) (1) {};
\node[vertex] at (0,1) (2) {};
\node[vertex] at (3,4) (3) {};
\node[vertex] at (4,3) (4) {};
\node[vertex] at (0,3) (5) {};
\node[vertex] at (1,4) (6) {};
\node[vertex] at (3,0) (7) {};
\node[vertex] at (4,1) (8) {};
\node[vertex] at (2,2.5) (9) {};
\node[vertex] at (1.6,1.8) (10) {};
\node[vertex] at (2.4,1.8) (11) {};

\filldraw[fill opacity=0,fill=white!70] 
($(1)+(0,-0.2)$) to ($(7)+(0,-0.2)$) to[out=0,in=270] ($(8)+(0.2,0)$) to ($(4)+(0.2,0)$) to[out=90,in=0] ($(3)+(0,0.2)$) to[out=180,in=45] ($(3)+(-0.1414,0.1414)$) to ($(2)+(-0.1414,0.1414)$) to[out=225,in=90] ($(2)+(-0.2,0)$) to[out=270,in=180] ($(1)+(0,-0.2)$);

\filldraw[fill opacity=0,fill=white!70] 
($(6)+(0,0.4)$) to ($(3)+(0,0.4)$) to[out=0,in=90] ($(4)+(0.4,0)$) to ($(8)+(0.4,0)$) to[out=270,in=0] ($(7)+(0,-0.4)$) to[out=180,in=315] ($(7)+(-0.2828,-0.2828)$) to ($(5)+(-0.2828,-0.2828)$) to[out=135,in=270] ($(5)+(-0.4,0)$) to[out=90,in=180] ($(6)+(0,0.4)$);

\filldraw[fill opacity=0,fill=white!70] 
($(3)+(0,0.6)$) to ($(6)+(0,0.6)$) to[out=180,in=90] ($(5)+(-0.6,0)$) to ($(2)+(-0.6,0)$) to[out=270,in=180] ($(1)+(0,-0.6)$) to[out=0,in=225] ($(1)+(0.4242,-0.4242)$) to ($(4)+(0.4242,-0.4242)$) to[out=45,in=270] ($(4)+(0.6,0)$) to[out=90,in=0] ($(3)+(0,0.6)$);

\filldraw[fill opacity=0,fill=white!70] 
($(7)+(0,-0.8)$) to ($(1)+(0,-0.8)$) to[out=180,in=270] ($(2)+(-0.8,0)$) to ($(5)+(-0.8,0)$) to[out=90,in=180] ($(6)+(0,0.8)$) to[out=0,in=135] ($(6)+(0.5656,0.5656)$) to ($(8)+(0.5656,0.5656)$) to[out=315,in=90] ($(8)+(0.8,0)$) to[out=270,in=0] ($(7)+(0,-0.8)$);

\node at (2,-1.5) {A $(9,4,2)$-link};
\end{tikzpicture}
\caption{Further examples of our hypergraphs.}
\end{figure}

Given $S$ is any of our defined hypergraphs, we say that $e_t$ \emph{completes} $S$ if the edge order $e_0,\dots,e_{t-1}$ does not contain a copy of $S$, whereas $e_0,\dots,e_t$ does.

Note that all bicycles have a valid edge order which contains at least two bad edges; we will in fact show that any hypergraph which has a valid edge order with at least two bad edges must contain a bicycle (see Claim~\ref{dlc2}). Meanwhile Pasch configurations, $(k,u,2)$-stars with $u \geq 2$ and $(k,u,a)$-links with $u\geq 3$ and $a \leq \floor{k/(u-1)}$ all have the property that any allowed edge order contains at least one $k$-bad edge, and also precisely one bad edge; in particular these hypergraphs do not contain a bicycle. 

The roles of bicycles are crucial in our proof. The deterministic lemmas will imply that Breaker has a winning strategy for the game played on any component of $H$ which does not contain a bicycle. The probabilistic lemma will show that w.h.p. $H$ does not contain any bicycles. \\

{\bf \ref{Q2} Two deterministic lemmas.}
Recall that we wish to show that if Maker wins the game on the associated hypergraph $H$, then $H$ must contain a particular subhypergraph. (In particular, this subhypergraph will be a bicycle.) This section contains two deterministic lemmas, which together prove that the contrapositive statement holds; that is, if $H$ does not contain a bicycle, then Breaker has a strategy to win the game on $H$.

\begin{lemma}\label{dl1}
Let $H'$ be a connected component of $H$ and suppose $H'$ does not contain a bicycle. Then $H'$ has an edge order $e_0,\dots,e_t$ with the property that there exists $a \in [0,t]$ such that $e_i$ is good for all $i \in [a+1,t]$, and also precisely one of the following holds:
\begin{itemize}
\item[(i)]{$a=0$;}
\item[(ii)]{$a \geq 2$ and $e_0,\dots,e_a$ forms a loose cycle;}
\item[(iii)]{$a=1$ and $e_0,e_1$ forms an overlapping pair;}
\item[(iv)]{$a=3$, $k=3$ and $e_0,\dots,e_3$ forms a Pasch configuration;}
\item[(v)]{$a \geq 3$ is odd and $e_0,\dots,e_a$ forms a $(k,(a+1)/2,2)$-star;}
\item[(vi)]{$a \geq 2$, and $e_0,\dots,e_a$ forms a $(k,(a+1),d)$-link, where $d \leq \floor{k/a}$.}
\end{itemize}
\end{lemma}

\begin{lemma}\label{dl2}
Let $H'$ be a component of $H$ which is as described in Lemma~\ref{dl1}. Breaker has a strategy for winning the Maker-Breaker game played on $H'$.
\end{lemma}

Note that by Breaker always choosing a vertex from the same component as Maker if he can, these results imply that if $H$ does not contain a bicycle, then Breaker can win the game played on $H$, and therefore the $(B,b')$-game on $[n]_p$.

We now prove four claims; the proof of Lemma~\ref{dl1} will follow easily from the statements of these claims.

\begin{claim}\label{dlc1} 
Suppose that $E_1:=e_0,\dots,e_t$ is an allowed edge order of the edges of a connected hypergraph $J$, for which $e_i$ for some $i \in [t]$ is the first bad or $k$-bad edge. Then we have the following:
\begin{itemize}
\item[(i)]{Either there exists $j \in [0,i-1]$ such that $e_j,e_i$ forms an overlapping pair, or $e_i$ completes a loose cycle.}
\item[(ii)]{Suppose that $S$ is a connected subhypergraph of $J$ with $s$ edges, which contains the overlapping pair or loose cycle guaranteed by (i). Then there exists an allowed edge order $E_2$ of $E(J)$ which starts with the overlapping pair or loose cycle, followed by the rest of the edges of $S$, followed by any remaining edges of $J$.}
\item[(iii)]{If $E_1$ is valid and if $S$ in (ii) is a loose cycle or overlapping pair, then it is possible to construct $E_2$ in (ii) so that additionally it is valid.}
\end{itemize}
\end{claim}

\begin{proof}
For (i), if there exists $j \in [0,i-1]$ such that $|e_j \cap e_i| \geq 2$ then $e_j,e_i$ forms an overlapping pair and we are done. So suppose: 
\begin{enumerate}[label=(A\arabic*)]
\item{For all $j \in [0,i-1]$ we have $|e_j \cap e_i| \leq 1$.}\label{da1}
\end{enumerate}
Also note that since $e_i$ is the first bad or $k$-bad edge in $E_1$, we have:
\begin{enumerate}[label=(A\arabic*),resume]
\item{For all $j \in [1,i-1]$, $e_j$ is good in $E_1$.}\label{da2}
\end{enumerate}
Since $e_i$ is bad or $k$-bad, it has $q \geq 2$ old vertices in the edge order $E_1$. Label these as $x_1,\dots,x_q$ and consider a minimal path $P_1:=e_{f_1},\dots,e_{f_u}$ in $\{e_0,\dots,e_{i-1}\}$ from $X_1:=\{x_1\}$ to $X_2:=\{x_2,\dots,x_q\}$. By definition of $P_1$, \ref{da1} and \ref{da2} we have
\begin{itemize}
\item{$|e_{f_j} \cap e_i| = \begin{cases} 1 \text{ if $j=1$ or $j=u$;} \\ 0 \text{ otherwise;} \end{cases}$}
\item{$u \geq 2$;}
\item{$P_1$ is a loose path.}
\end{itemize}
It follows from these three facts that $e_{f_1},\dots,e_{f_u},e_i$ forms a loose cycle. By \ref{da2} $e_0,\dots,e_{i-1}$ clearly does not contain a loose cycle, and hence $e_i$ completes a loose cycle in $e_0,\dots,e_i$.

For (ii), such an allowed edge order $E_2$ exists since both $S$ and $J$ are connected; simply pick the overlapping pair or loose cycle first, then pick the remaining edges of $S$ in any way so that each edge has non-empty intersection with the set of all previously chosen edges. Then pick the remaining edges of $J$ in the same way. 

For (iii), suppose $E_1$ is valid and $S$ is an overlapping pair or loose cycle. Then by the definition of $E_1$, there are $k$ new vertices in $e_0$, $k-1$ new vertices in $e_a$ for each $a \in [i-1]$ and at least one new vertex in $e_i$. Thus the hypergraph $J':=e_0,\dots,e_i$ satisfies
\begin{enumerate}[label=(A\arabic*),resume]
\item{$|V(J')| \geq k+(k-1)(i-1)+1$.}\label{da3}
\end{enumerate}
If there was an allowed edge order $E_3$ of $E(J')$ which contained a $k$-bad edge, we would have $|V(J')| \leq k+(k-1)(i-1)$ since $E_3$ has one initial edge, at most $i-1$ good or bad edges, and at least one $k$-bad edge. However this violates \ref{da3} and thus:
\begin{enumerate}[label=(A\arabic*),resume]
\item{All allowed edge orders of $E(J')$ are valid.}\label{da4}
\end{enumerate}
Now consider the edge order $E_2$ which starts with the loose cycle or overlapping pair, followed by the rest of the edges in $\{e_0,\dots,e_i\}$ chosen so that each edge has non-empty intersection with the set of all previously chosen edges, then followed by $e_{i+1},\dots,e_t$ (in this order). First note that for any $e_a$ with $a \in [i+1,t]$, the set of all previously chosen edges is $\{e_0,\dots,e_{a-1}\}$ in both edge orders $E_1$ and $E_2$. Thus if $e_a$ is good or bad in $E_1$, then it is also good or bad in $E_2$. Finally note that by \ref{da4}, the first $i+1$ edges in $E_2$ form a valid edge order of $E(J')$. We conclude that $E_2$ is the valid edge order of $E(J)$ required. 
\end{proof}

\begin{claim}\label{dlc2}
A hypergraph $J$ does not contain a bicycle if and only if any valid edge order of the edges of any connected subhypergraph $J'$ of $J$ has at most one bad edge.
\end{claim}

\begin{proof}
First note that if $J$ contains a bicycle, then by considering the edge order $e_{f_1},\dots,e_{f_u}$ given in the definitions of each of the hypergraphs which the bicycle could be, we see immediately that $J$ contains a connected subhypergraph which has a valid edge order with at least two bad edges.

Now we must show that if there exists a connected subhypergraph $J'$ of $J$ and a valid edge order of $E(J')$ with at least two bad edges, then $J$ contains a bicycle. 

So let $J'$ be such a hypergraph, and let $E_1:=e_0,\dots,e_t$ be the valid edge order of $E(J')$ with at least two bad edges. By using all three parts of Claim~\ref{dlc1}, we may assume without loss of generality that there exists $i \in [t]$ such that we have precisely one of the following: 
\begin{enumerate}[label=(B\arabic*)]
\item{$i=1$ and $e_0,e_1$ forms an overlapping pair;}\label{db1}
\item{$e_0,\dots,e_i$ forms a loose cycle (with edges ordered cyclically).}\label{db2}
\end{enumerate}
Let $P_1:=e_0,\dots,e_i$ and let $j \in [i+1,t]$ be such that $e_j$ is the next bad edge in $E_1$ after $e_i$. We have precisely one of the following:
\begin{enumerate}[label=(B\arabic*),resume]
\item{$2 \leq |e_j \cap V(P_1)| \leq k-1$;}\label{db3}
\item{There exists $a \in [i+1,j-1]$ such that $|e_j \cap e_a| \geq 2$ and $x=e_j \cap V(P_1)$ and $y=e_a \cap V(P_1)$ are distinct vertices;}\label{db4}
\item{There exists $a \in [i+1,j-1]$ such that $|e_j \cap e_a| \geq 2$ and $|(e_a \cup e_j) \cap V(P_1)| =1$;}\label{db5}
\item{There exists $a \in [i+1,j-1]$ such that $|e_j \cap e_a| \geq 2$ and $(e_a \cup e_j) \cap V(P_1) = \emptyset$;}\label{db6}
\item{For all $a \in [0,j-1]$ we have $|e_j \cap e_a| \leq 1$, and $|e_j \cap V(P_1)|=1$;}\label{db7}
\item{For all $a \in [0,j-1]$ we have $|e_j \cap e_a| \leq 1$, and $e_j \cap V(P_1) = \emptyset$.}\label{db8}
\end{enumerate}
For each case, it suffices to find a subhypergraph of $J'$ which is a bicycle. Throughout we will make use of the following fact:
\begin{enumerate}[label=(B\arabic*),resume]
\item{For all $a \in [j-1] \setminus \{i\}$, $e_a$ is good in $E_1$.}\label{db9}
\end{enumerate}

{\bf Case 1: \ref{db3} holds.}
Clearly $e_0,\dots,e_i,e_j$ forms an overlapping pair/loose cycle with handle.

{\bf Case 2: \ref{db4} holds.}
We have precisely one of the following:
\begin{itemize}
\item{There exists $d \in [0,i]$ such that $x,y \in e_d$; then $e_j,e_a,e_d$ forms an overlapping pair with handle.}
\item{There does not exist $d \in [0,i]$ such that $x,y \in e_d$; without loss of generality suppose that $x \in e_0 \setminus e_1$. If $P_1$ is an overlapping pair, then $e_0,e_1,e_a,e_j$ forms a double overlapping pair. (Note that $|e_a \cap e_j| \leq k-2$ since otherwise $e_j$ would be $k$-bad in the edge order $E_1$.) If $P_1$ is a loose cycle, then let $d \in [i]$ be the smallest integer such that $y \in e_d$. If $d=i$, then $e_a,e_j,e_0,e_i$ forms a spoiled cycle; otherwise $e_a,e_j,e_0,\dots,e_d$ forms a spoiled cycle.}
\end{itemize}

{\bf Case 3: \ref{db5} holds.}
Let $P_2:=e_a,e_j$ and note that $|V(P_2) \cap V(P_1)|=1$. It follows that $e_0,\dots,e_i,e_a,e_j$ forms an overlapping pair/loose cycle to overlapping pair.

{\bf Case 4: \ref{db6} holds.}
Again let $P_2:=e_a,e_j$ and note that we have $V(P_2) \cap V(P_1)=\emptyset$. So consider a minimal path $P_3:=e_{f_1},\dots,e_{f_u}$ in $\{e_{i+1},\dots,e_{j-1}\} \setminus e_a$ from $X_1:=V(P_1)$ to $X_2:=V(P_2)$. By \ref{db9} $P_3$ is a loose path, moreover by definition of $P_3$, we have $|V(P_1) \cap V(P_3)|=1$ and $|V(P_2) \cap V(P_3)|=1$. Additionally if $u \geq 2$, then $e_{f_2} \cap V(P_1) = \emptyset$ and $e_{f_{u-1}} \cap V(P_2) = \emptyset$.
Thus $P_1$, $P_3$ and $P_2$ together form an overlapping pair/loose cycle to overlapping pair.

{\bf Case 5: \ref{db7} holds.}
Since $e_j$ has at least one old vertex which is not in $e_0,\dots,e_i$, we may consider a minimal path $P_2:=e_{f_1},\dots,e_{f_u}$ in $\{e_{i+1},\dots,e_{j-1}\}$ from $X_1:=V(P_1)$ to $X_2:=e_j \setminus V(P_1)$. First note by \ref{db9} that $P_2$ is a loose path. 
Now we have precisely one of the following:
\begin{itemize}
\item{We have $(e_{f_1} \cap e_j \cap V(P_1))\not= \emptyset$; then $P_3:=e_{f_1},\dots,e_{f_u},e_j$ forms a loose cycle. Now since $|V(P_1) \cap V(P_3)|=1$, $P_1$ and $P_3$ together form an overlapping pair/loose cycle to loose cycle.}
\item{There exists $d \in [0,i]$ such that $x=e_j \cap e_d$ and $y=e_{f_1} \cap e_d$ are distinct vertices; then $P_3:=e_d,e_{f_1},\dots,e_{f_u},e_j$ forms a loose cycle. If $P_1$ is an overlapping pair, then let $d' \in \{0,1\}$ be such that $d' \not=d$. Then $P_3$ together with $e_{d'}$ forms a loose cycle with handle. If $P_1$ is a loose cycle, then let $P_4$ be the loose path $e_{f_1},\dots,e_{f_u},e_j$ and define $z:=j$ if $u=1$ and $z:=f_2$ otherwise. Observe that $P_1$ and $P_4$ are vertex-disjoint except for $x= (e_j \setminus e_{f_u}) \cap e_d$ and $y=(e_{f_1} \setminus e_{z}) \cap e_d$, and thus together form a double loose cycle.}
\item{We have $x=e_j \cap V(P_1)$ and $y=e_{f_1} \cap V(P_1)$ are not together in any edge of $P_1$; without loss of generality suppose that $x \in e_0 \setminus e_1$. If $P_1$ is an overlapping pair, then $e_0,e_1,e_{f_1},\dots,e_{f_u},e_j$ forms a spoiled cycle. If $P_1$ is a loose cycle, then define $P_4,z$ as in the previous bullet point. Then observe that $P_1$ and $P_4$ are vertex-disjoint except for $x= (e_j \setminus e_{f_u}) \cap e_0$ and $y=(e_{f_1} \setminus e_{z}) \cap e_d$ for some $d \in [i]$, and thus as before, form a double loose cycle.
}
\end{itemize}

{\bf Case 6: \ref{db8} holds.}
Label the $q \geq 2$ old vertices of $e_j$ as $x_1,\dots,x_q$ and consider a minimal path $P_2:=e_{f_1},\dots,e_{f_u}$ in $\{e_{i+1},\dots,e_{j-1}\}$ from $X_1:=V(P_1)$ to $X_2:=\{x_1,\dots,x_q\}$. Let $x_a:=e_{f_u} \cap e_j$ and consider a minimal path $P_3:=e_{f_{u+1}}, \dots, e_{f_v}$ in $\{e_{i+1},\dots,e_{j-1}\} \setminus E(P_2)$ from $X_3:=X_2 \setminus x_a$ to $X_4:=V(P_1) \cup V(P_2)$. By \ref{db9} both $P_2$ and $P_3$ are loose paths. Now we have precisely one of the following:
\begin{itemize}
\item{We have $e_{f_v} \cap V(P_1) = \emptyset$; let $d \in [u]$ be the largest integer such that $e_{f_d} \cap e_{f_v} \not= \emptyset$. Then $P_4:=e_{f_d},\dots,e_{f_u},e_j,e_{f_{u+1}},\dots,e_{f_v}$ forms a loose cycle. If $d=1$, then since $|V(P_1) \cap V(P_4)|=1$, we have that $P_1$ and $P_4$ together form an overlapping pair/loose cycle to loose cycle. Otherwise let $P_5:=e_{f_1},\dots,e_{f_{d-1}}$. Then we have $|V(P_1) \cap V(P_5)|=1$, $|V(P_4) \cap V(P_5)|=1$ and $V(P_1) \cap V(P_4) = \emptyset$. Additionally, if $d \geq 3$, then we see that $e_{f_2} \cap V(P_1) = \emptyset$ and $e_{f_{d-2}} \cap V(P_4)= \emptyset$. Thus $P_1, P_4, P_5$ form an overlapping pair/loose cycle to loose cycle.}
\item{We have $e_{f_v} \cap V(P_1) \not= \emptyset$; the properties of the hypergraph $e_0,\dots,e_i,e_{f_1},\dots,e_{f_u},e_j,e_{f_{u+1}},\dots,e_{f_v}$ are identical to that the hypergraph $e_0,\dots,e_i,e_{f_1},\dots,e_{f_u},e_j$ found in Case 5 (up to the labelling of the edges), so a similar case study yields a bicycle.}
\end{itemize}
\end{proof}

For the remainer of the proof, we shall call a valid edge order which contains at least two bad edges a \emph{bad edge order}. If a hypergraph $J$ does not contain a bicycle, then by Claim~\ref{dlc2}, the existence of a bad edge order of $E(J')$ where $J'$ is a subhypergraph of $J$ is a contradiction. In the claims which follow we will always assume that $J$ does not contain a bicycle, and hence whenever some assumed condition of a case within a case analysis leads to the discovery of a bad edge order, we can immediately stop and move onto the next case.

\begin{claim}\label{dlc3}
Let $S$ be a hypergraph with $s$ edges, which is an overlapping pair, a loose cycle, a $(k,s/2,2)$-star with $s \geq 4$, a $(k,s,a)$-link with $s \geq 3$ and $a \in [\floor{k/(s-1)}]$, or a Pasch configuration.
Suppose $J$ is a connected hypergraph which does not contain a bicycle, and does contain $S$.
Then there exists an allowed edge order $E_1:=e_0,\dots,e_t$ such that
\begin{itemize}
\item[(i)]{$e_0,\dots,e_i$ forms an overlapping pair or loose cycle, $e_0,\dots,e_{s-1}$ forms $S$, and every edge $e_j$ for $j \in [s,t]$ is either good or $k$-bad;} 
\item[(ii)]{For all $j \in [s+1,t]$, if $e_j$ is $k$-bad, then either $e_{j-1}$ is also $k$-bad, or there exists a vertex $x \in e_j$ which is new in $e_{j-1}$.}
\end{itemize}
\end{claim}

\begin{proof}
For (i), by Claim~\ref{dlc1}(ii) we can assume that the edge order starts with the overlapping pair or loose cycle, followed by the rest of the edges of $S$, and that $e_i$ is bad. If there exists another bad edge $e_j$ for some $j \in [s,t]$, then the edge order $E_2$, found by deleting from $e_0,\dots,e_j$ all $k$-bad edges, is bad. Thus for all $j \in [s,t]$, $e_j$ must either be good or $k$-bad. 

For (ii) suppose that $E_1=e_0,\dots,e_t$ does not satisfy the property stated in (ii). Then we have the following:
\begin{enumerate}[label=(C\arabic*),resume]
\item{There exists $j \in [s+1,t]$ such that $e_{j-1}$ is good and $e_j$ is $k$-bad, and all vertices in $e_j$ appeared in the edge order before $e_{j-1}$.}\label{dc1}
\end{enumerate} 
Now consider the edge order $E_2:=e_0,\dots,e_{j-2},e_j,e_{j-1},e_{j+1},\dots,e_t$. In this order, $e_j$ is still $k$-bad and $e_{j-1}$ is still good; moreover $E_2$ still starts with the overlapping pair or loose cycle. Hence by continuously performing swaps whenever such a pair $e_{j-1}$ and $e_j$ exists (satisfying \ref{dc1}), we eventually reach an edge order $E_p$ where no such pair exists. Thus in the final edge order $E_p$ the property stated in (ii) holds.
\end{proof}

\begin{claim}\label{dlc4}
Let $S$ be a hypergraph with $s$ edges, which is an overlapping pair, a loose cycle, a $(k,s/2,2)$-star with $s \geq 4$, a $(k,s,a)$-link with $s \geq 3$ and $a \in [\floor{k/(s-1)}]$, or a Pasch configuration.
Suppose $J$ is a connected hypergraph which does not contain a bicycle, and does contain $S$. Finally suppose $E_1:=e_0,\dots,e_t$ is the allowed edge order of $E(J)$ guaranteed by Claim~\ref{dlc3}, which starts with the edges of $S$, in particular with the overlapping pair or loose cycle $P_1:=e_0,\dots,e_i$. Suppose that $E_1$ contains at least one $k$-bad edge amongst the edges $e_s,\dots,e_t$. Then we have precisely one of the following:
\begin{itemize}
\item[(i)]{We have that $e_0,e_1,e_{j-1},e_j$ forms a $(k,2,2)$-star for some $j \in [s+1,t]$. Moreover either $S$ is an overlapping pair, or $S$ is a $(k,s/2,2)$-star with $s \geq 4$ and $e_0,\dots,e_{s-1},e_{j-1},e_j$ forms a $(k,s/2+1,2)$-star.}
\item[(ii)]{We have that $e_0,e_1,e_s$ forms a $(k,3,a)$-link, where $a=|e_0 \setminus e_1|$. Moreover either $S$ is an overlapping pair, or $S$ is a $(k,s,a)$-link with $s \geq 3$ and $e_0,\dots,e_s$ forms a $(k,s+1,a)$-link.}
\item[(iii)]{We have that $S$ is a loose cycle with three edges, $k=3$, and $e_0,e_1,e_2,e_3$ forms a Pasch configuration.}
\end{itemize}
\end{claim}

\begin{proof}
Suppose that $e_j$ is the first $k$-bad edge amongst the edges $e_s,\dots,e_t$ (so if $j>s$ then $e_s,\dots,e_{j-1}$ are good). 
Let $e_j:=\{x_1,\dots,x_k\}$ and let $e_{f_1},\dots,e_{f_k}$ be the respective edges in which each $x_i$ is new in $E_1$, noting that without loss of generality we have $f_1 \leq \cdots \leq f_k$.

We have precisely one of the following:
\begin{enumerate}[label=(D\arabic*)]
\item{We have $j \in [s+1,t]$;}\label{ddd1}
\item{We have $j=s$ and $P_1=e_0,e_1$ is an overlapping pair;}\label{ddd2}
\item{We have $j=s$ and $P_1=e_0,\dots,e_i$ is a loose cycle.}\label{ddd3}
\end{enumerate}
We will go through each of these cases in turn and show that Claims~\ref{dlc4}(i), (ii) and (iii) hold respectively.

{\bf Case 1: \ref{ddd1} holds.}
Without loss of generality, we have precisely one of the following:
\begin{enumerate}[label=(D\arabic*),resume]
\item{We have $2 \leq |e_j \cap V(P_1)| \leq k-1$;}\label{dd1}
\item{We have $|e_j \cap V(P_1)|  \leq 1$ and for all $a \in [i+1,j-1]$ we have $|e_j \cap e_a| \leq 1$;}\label{dd2}
\item{There exists $a \in [i+1,j-1]$ such that $|e_j \cap e_a| \geq 2$ and $|(e_j \cup e_a) \cap V(P_1)| \leq 1$;}\label{dd3}
\item{There exists $a \in [i+1,j-1]$ such that $|e_j \cap e_a| \geq 2$ and $x=e_j \cap V(P_1)$ and $y=e_a \cap V(P_1)$ are distinct vertices.}\label{dd4}
\end{enumerate}

{\bf Case 1a: \ref{dd1} holds.} 
We have that $e_0,\dots,e_i,e_a$ forms an overlapping pair/loose cycle with handle, a contradiction to $J$ not containing a bicycle. 

{\bf Case 1b: \ref{dd2} holds.}
Since we have $|V(P_1) \cap e_j| \leq 1$ and each $e_a$ for $a \in [i+1,j-1]$ is good in $E_1$, we have
\begin{enumerate}[label=(D\arabic*),resume]
\item{$f_2 >i$;}\label{dd5}
\item{$f_1 < \cdots < f_k$.}\label{dd6}
\end{enumerate}
Then consider the edge orders $E_2:=e_0,\dots,e_{f_2},e_j$ and $E_3$, which is formed by deleting from $E_2$ each edge which is $k$-bad. We will show that $E_3$ is a bad edge order. 
First note that clearly $E_3$ is a valid edge order since all $k$-bad edges were deleted. 
By \ref{dd5}, $E_2$ and hence also $E_3$ both start with the edges of $P_1$, so have at least one bad edge. Further $e_j$ is bad in $E_2$ since precisely two of the vertices in $e_j$ are old in $E_2$, namely $x_1$ and $x_2$, which appear in $e_{f_1}$ and $e_{f_2}$ respectively. Since $e_{f_1}$ and $e_{f_2}$ both contain a new vertex, they are not $k$-bad, and so are both contained in $E_3$. Thus $e_j$ is also bad in $E_3$, and so $E_3$ is indeed a bad edge order.

{\bf Case 1c: \ref{dd3} holds.}
Let $P_2:=e_a,e_j$ and note that if $|V(P_1) \cap V(P_2)|=1$, then $V(P_1)$ and $V(P_2)$ together form an overlapping pair/loose cycle to overlapping pair, a contradiction to $J$ not containing a bicycle. Otherwise we have $V(P_1) \cap V(P_2) = \emptyset$, so consider a minimal path $P_3:=e_{g_1},\dots,e_{g_u}$ in $\{e_d: d\in [i+1,j-1], d \not= a, e_d$ is good in $E_1\}$
from $X_1:=V(P_1)$ to $X_2:=V(P_2)$. By the choice of where the edges in $P_3$ are selected from, $P_3$ is a loose path. Additionally we have $|V(P_1) \cap V(P_3)|=1$, $|V(P_2) \cap V(P_3)|=1$, and if $u \geq 2$, then $V(P_1) \cap e_{g_2} = \emptyset$ and $V(P_2) \cap e_{g_{u-1}} = \emptyset$. Thus $P_1,P_2,P_3$ form an overlapping pair/loose cycle to overlapping pair, a contradiction to $J$ not containing a bicycle. 

{\bf Case 1d: \ref{dd4} holds.}
First suppose that $P_1$ is a loose cycle and without loss of generality that $x \in e_0 \setminus e_i$. If $y \in e_0$, then let $E_2:=e_a,e_j,e_0$. If $y \in e_i \setminus e_0$, then let $E_2:=e_a,e_j,e_0,e_i$. Otherwise let $E_2:=e_a,e_j,e_0,\dots,e_i$. For each case $e_j$ and the last edge are both bad edges in $E_2$. Further it is easy to see that in each case, $E_2$ is valid, and thus $E_2$ is bad.

Now suppose that $P_1$ is an overlapping pair. If $x$ and $y$ are both in $e_d$ for $d=0$ or $d=1$, then $e_j,e_a,e_d$ is a bad edge order. So suppose without loss of generality that $x \in e_1 \setminus e_0$, and $y \in e_0 \setminus e_1$. If $|e_j \cap e_a| \leq k-2$, then $e_0,e_1,e_j,e_a$ forms a double overlapping pair. Similarly if $|e_0 \cap e_1| \leq k-2$, then $e_j,e_a,e_0,e_1$ forms a double overlapping pair. Both would contradict $J$ not containing a bicycle, and thus we have $|e_j \cap e_a| = |e_0 \cap e_1| = k-1$. Hence we also have $x=e_1 \setminus e_0=e_j \setminus e_a$ and $y=e_0 \setminus e_1=e_a \setminus e_j$ and so in particular, we have that $e_0,e_1,e_j,e_a$ forms a $(k,2,2)$-star. \\

The conclusion of our case analysis is that there exists $a \in [2,j-1]$ such that $e_0,e_1,e_a,e_j$ forms a $(k,2,2)$-star and that $P_1$ is an overlapping pair where $x=e_1 \setminus e_0$ and $y=e_0 \setminus e_1$.

Since $E_1$ is the edge order obtained from Claim~\ref{dlc3}, we have $e_j$ comes immediately after the edge for which the last of the vertices of $e_j$ are new, or following another $k$-bad edge. Therefore, since $e_{j-1}$ is not $k$-bad and $e_j \subseteq (e_0 \cup e_1 \cup e_a)$, we conclude that $a=j-1$. 

Note that $S$ cannot be a loose cycle or Pasch configuration, since $P_1$ is an overlapping pair. If $S$ is a $(k,s,d)$-link for some $s \geq 3$, then we have $e_2 \cap (e_{j-1} \cup e_j) = \{x,y\}$ and so the edge order $e_{j-1},e_j,e_2$ is bad. If $S$ is an overlapping pair, then there is nothing further to prove. Finally if $S$ is a $(k,s/2,2)$-star, then clearly $e_0,\dots,e_{s-1},e_{j-1},e_j$ forms a $(k,s/2+1,2)$-star (with central vertices $x$ and $y$). \\

{\bf Case 2: \ref{ddd2} holds.}
First note that $S$ cannot be a loose cycle or Pasch configuration, since $P_1$ is an overlapping pair. We have precisely one of the following:
\begin{enumerate}[label=(D\arabic*),resume]
\item{We have $|e_s \cap (e_0 \cup e_1)|=k$;}\label{dde1}
\item{We have $2 \leq |e_s \cap (e_0 \cup e_1)| \leq k-1$;}\label{dde2}
\item{We have $|e_s \cap (e_0 \cup e_1)| \leq 1$.}\label{dde3}
\end{enumerate}

{\bf Case 2a: \ref{dde1} holds.} 
First suppose that $|e_s \cap e_1| \leq 1$. Since $e_0,e_1$ is an overlapping pair, we have $$k=|e_s|=|e_s \cap e_1| + |e_s \cap (e_0 \setminus e_1) | \leq k-1,$$ a contradiction, and hence we must have $|e_s \cap e_1| \geq 2$. Similarly $|e_s \cap e_0| \geq 2$. Thus any permutation of the edges $e_0,e_1,e_s$ must have that the second edge is bad and the third edge is bad or $k$-bad. In order to not obtain a bad edge oder, we must have that the third edge is $k$-bad in all of these permutations. Thus by definition $e_0,e_1,e_s$ forms a $(k,3,|e_0 \setminus e_1|)$-link. If $S$ is an overlapping pair, then there is nothing further to prove. If $S$ is a $(k,s/2,2)$-star for some $s \geq 4$, then without loss of generality we have $x=e_0 \setminus e_1=e_2 \setminus e_3$ and $y=e_1 \setminus e_0 = e_3 \setminus e_2$. But then since $x,y \in e_s$, the edge order $e_2,e_3,e_s$ is bad. Finally suppose that $S$ is a $(k,s,a)$-link. Then for all $d,d' \in [0,s-1]$ with $d<d'$, we have that $e_d, e_{d'}$ forms an overlapping pair, and $|e_d \setminus e_d'|=|e_0 \setminus e_1|=a$. By repeating the argument above for the permutations of $e_0,e_1,e_s$, we see that any permutation of $e_d,e_{d'},e_s$ must have that the second edge is bad and the third is $k$-bad. Thus $e_0,\dots,e_s$ forms a $(k,s+1,|e_0 \setminus e_1|)$-link. 

{\bf Case 2b: \ref{dde2} holds.}
We have that $e_0,e_1,e_s$ forms an overlapping pair with handle, a contradiction to $J$ not containing a bicycle.

{\bf Case 2c: \ref{dde3} holds.}
First note that since $e_s$ contains vertices outside of $e_0 \cup e_1$, we have that $S$ cannot be an overlapping pair or a  $(k,s,|e_0 \setminus e_1|)$-link, and thus $S$ must be a $(k,s/2,2)$-star. If there exists $i,j$ such that $|e_i \cap e_j| =k-1$ and $|e_s \cap (e_i \cup e_j)| \geq 2$, then repeat the argument from Case 2a or 2b with $e_0$ and $e_1$ replaced by $e_i$ and $e_j$. For the remaining case we have that for all $i,j$ such that $|e_i \cap e_j| =k-1$, we have  $|e_s \cap (e_i \cup e_j)| \leq 1$ and in particular, the central vertices of the two stars are not in $e_s$. Now without loss of generality, suppose $e_2$ and $e_3$ is an overlapping pair, $|e_s \cap (e_0 \cap e_1)| =1$ and $|e_s \cap (e_2 \cap e_3)| =1$. Then the edge order $e_0,e_1,e_2,e_s$ is a bad edge order. \\

{\bf Case 3: \ref{ddd3} holds.}
Since $S$ contains a loose cycle, $S$ cannot be an overlapping pair, $(k,s/2,2)$-star or a $(k,s,a)$-link, and hence $S$ is a loose cycle or Pasch configuration. If $S$ is the latter, then we have $j=4$, and by the symmetry of the Pasch configuration, we have without loss of generality that $|e_0 \cap e_4|=2$ and $|e_1 \cap (e_0 \cup e_4)|=2$. Hence $e_0,e_4,e_1$ is a bad edge order. Thus $S$ must be a loose cycle and $j=i+1$. Without loss of generality we have precisely one of the following:
\begin{itemize}
\item{$|e_0 \cap e_{i+1}| \geq 2$ and $(e_i \setminus e_0) \cap e_{i+1} = \emptyset$; let $E_2:=e_{i+1},e_0,\dots,e_i$. 
For all $a \in [i]$, we have $|(e_a \setminus e_{a-1}) \cap e_{i+1}| \leq k-2$, and thus
$$|e_a \cap (e_{a-1} \cup e_{i+1})| \leq |e_{a-1} \cap e_a| + | (e_a \setminus e_{a-1}) \cap e_{i+1}| \leq k-1.$$ Thus $E_2$ is a valid edge order. Further $e_0$ is bad in $E_2$. Since $e_i$ completes the cycle $e_0,\dots,e_i$, it is also bad in $E_2$, and thus $E_2$ is a bad edge order.}
\item{$|e_0 \cap e_{i+1}| \geq 2$ and $(e_i \setminus e_0) \cap e_{i+1} \not= \emptyset$; let $E_2:=e_{i+1},e_0,e_i$. Here we have $$2 \leq |e_i \cap (e_{i+1} \cup e_0)| \leq k-1$$ and so $E_2$ is a bad edge order.}
\item{For all $a \in [0,i]$, we have $|e_{i+1} \cap e_a| \leq 1$ and $(e_i \setminus (e_0 \cup e_{i-1}) ) \cap e_{i+1} = \emptyset$; We have that \ref{dd6} holds, and thus $E_2:=e_0,\dots,e_{f_2},e_{i+1},e_{f_2+1},\dots,e_i$ is a bad edge order.}
\item{For all $a \in [0,i]$, we have $|e_{i+1} \cap e_a| \leq 1$, $|(e_a \setminus \{e_d: d \in [0,i], d \not =a\}) \cap e_{i+1}| =1$ and $k \geq 4$; Then $E_2:=e_0,e_1,e_{i+1},e_2$ is a bad edge order.}
\item{For all $a \in [0,i]$, we have $|e_{i+1} \cap e_a| \leq 1$, $|(e_a \setminus \{e_d: d \in [0,i], d \not =a\}) \cap e_{i+1}| =1$ and $k=3$. Then we have $i=2$ and $e_0,e_1,e_2,e_3$ forms a Pasch configuration.}
\end{itemize}
Only the last case does not produce a bad edge order, and thus we have that $e_0,e_1,e_2,e_3$ forms a Pasch configuration, as required.
\end{proof}

We are now ready to prove our two deterministic lemmas. 

\begin{proof}[Proof of Lemma~\ref{dl1}]
If there exists a simple edge order of $E(H')$, then we have (i), so suppose that any edge order of $E(H')$ contains at least one bad or $k$-bad edge. Now by Claim~\ref{dlc1} we may assume that $E_1:=e_0,\dots,e_t$ is an allowed edge order of $E(H')$ which starts with an overlapping pair or loose cycle $P_1:=e_0,\dots,e_i$. Then using Claim~\ref{dlc3} applied with $S:=P_1$, we have that every edge $e_{i+1},\dots,e_t$ is either $k$-bad or good.
If all of these edges are good, then we have (ii) or (iii), so are done. So assume that there is at least one $k$-bad edge. Then using Claim~\ref{dlc4} applied with $S:=P_1$, we have that $H'$ contains
a $(k,2,2)$-star, a $(k,3,|e_0 \setminus e_1|)$-link or a Pasch configuration. We can now repeatedly use Claims~\ref{dlc3} and~\ref{dlc4} as follows:
\begin{itemize}
\item[(a)]{Let $S$ be the $(k,p,2)$-star, a $(k,p+1,|e_0 \setminus e_1|)$-link or a Pasch configuration found previously (where $p \geq 2$). By Claim~\ref{dlc3}, there exists an allowed edge order of $E(H')$ which starts with all of the edges of $S$. If there are no further $k$-bad edges, then we have (iv), (v) or (vi). If there are, then move to step (b).}
\item[(b)]{By Claim~\ref{dlc4}, either $S$ was a $(k,p,2)$-star and $H'$ contains a $(k,p+1,2)$-star, or $S$ was a $(k,p+1,|e_0 \setminus e_1|)$-link and $H'$ contains a $(k,p+2,|e_0 \setminus e_1|)$-link. Now return to step (a).}
\end{itemize}
Since $H'$ is a finite hypergraph, this process must eventually stop, and hence we have (iv), (v) or (vi).
\end{proof}

For the proof of Lemma~\ref{dl2}, we simply find an explicit strategy for Breaker to win the game played on $H'$. 

\begin{proof}[Proof of Lemma~\ref{dl2}]
By Lemma~\ref{dl1}, $H'$ may contain a subhypergraph $S$ for which there exists an edge order which starts with all of the edges of $S$, and all subsequent edges are good: The subhypergraph $S$ if it exists must be

{\bf Case 1:} A $(k,u,2)$-star for some integer $u \geq 2$;

{\bf Case 2:} A $(k,u,a)$-link for some $u,a \in \mathbb{N}$ with $u \geq 3$ and $a \leq \floor{k/(u-1)}$;

{\bf Case 3:} A Pasch configuration;

{\bf Case 4:} A loose cycle;

{\bf Case 5:} An overlapping pair.

Let the edges of $S$ be $e_0,\dots,e_{s-1}$, and the rest of the good edges $e_{s},\dots,e_t$. (If $S$ does not exist, then set $s=0$.) 

Breaker uses the following strategy. 

\begin{itemize}
\item{If Maker selects a vertex in $S$, then Breaker does the following, corresponding to the cases above for what $S$ could be.

{\bf Case 1:} There are $2u$ edges; suppose without loss of generality that they are labelled so that $e_i$ and $e_{i+u}$ have intersection $k-1$ for each $i \in [0,u-1]$. Suppose Maker has selected a vertex in $e_j \cap e_{j+u}$ for $j \in [0,u-1]$. Then Breaker if he can, also selects such a vertex. Otherwise he selects an arbitrary vertex.

{\bf Case 2:} Breaker selects an arbitrary vertex in $S$ if he can. Otherwise he selects an arbitrary vertex.

{\bf Case 3:} If Maker has two out of three vertices from one of the edges of $S$, Breaker chooses the final vertex from this edge. Otherwise Breaker chooses an arbitrary vertex in $S$ if he can. If he cannot then he chooses an arbitrary vertex.

{\bf Case 4:} Assume without loss of generality that the edges are ordered cyclically $e_0,\dots,e_{s-1}$. If Maker has selected an element from $e_i \setminus e_{j}$ where $i \in [s-1]$ and $j=i-1$, or $i=0$ and $j=s-1$, then Breaker if he can, also selects such a vertex. Otherwise he selects an arbitrary vertex.

{\bf Case 5:} These two edges are $e_0$ and $e_1$. If Maker has selected an element from $e_0 \cap e_1$, then Breaker if he can, also selects such a vertex. Otherwise he selects an arbitrary vertex.}

\item{If Maker selects any other vertex, let $i$ be such that $e_i$ is the edge in which this vertex is new. If Breaker can, he also selects a vertex which is new in $e_i$. Otherwise he selects an arbitrary vertex.}
\end{itemize}

We must now show that at the end of the game, Maker has failed to claim every vertex of any edge in $H'$. First observe that Maker has failed to claim any of the edges $e_{s},\dots,e_t$. Let $i \in [s,t]$. In the edge order, $e_i$ is good, and so there exists at least $k-1 \geq 2$ vertices $x_i$ and $y_i$ which are new in $e_i$ (they do not appear in any edge $e_j$ for $j<i$). By part two of the strategy above, Maker cannot claim all of the vertices of $e_i$ since as soon as she tries to claim one of the at least two new vertices in $e_i$, Breaker will claim another new vertex in $e_i$. Thus if Maker has won the game, she must have claimed an edge from $S$. However, we will now run through each case, corresponding to the cases for $S$ in Breaker's strategy above, showing that Maker has not claimed such an edge.

{\bf Case 1:} Let $j \in [0,u-1]$ and suppose Maker is trying to claim $e_j$ or $e_{j+u}$. There are $k-1 \geq 2$ vertices in $e_j \cap e_{j+u}$ and so Maker cannot claim all of the vertices of $e_j$ or $e_{j+u}$ since as soon as she tries to claim one of the vertices which lie in $e_j \cap e_{j+u}$, Breaker will claim another vertex in $e_j \cap e_{j+u}$. Since all edges of $S$ are of this form, Maker cannot claim any edge of $S$.

{\bf Case 2:} Note that a $(k,u,a)$-link has at most $2k-2$ vertices. Hence by Breaker always claiming any vertex in $S$ whenever Maker does, he ensures that Maker can claim at most $\ceil{(2k-2)/2}=k-1$ of the vertices in $S$, therefore does not have enough to claim a full edge of $S$.

{\bf Case 3:} Breaker always tries to claims a vertex in $S$ if Maker does, hence Maker claims at most three of the six vertices in $S$. Note that any pair of vertices in $S$ lie together in at most one edge. Hence by Breaker selecting the third vertex of an edge if Maker has selected the first two, Maker is never able to claim all three vertices of an edge of $S$.

{\bf Case 4:} Suppose Maker is trying to claim $e_i$ for some $i \in [0,s-1]$. Let $j=i-1$ if $i\geq 1$, and let $j=s-1$ if $i=0$. There are $k-1 \geq 2$ vertices in $e_i \setminus e_j$ and so Maker cannot claim all of the vertices of $e_i$ since as soon as she tries to claim one of the vertices in $e_i \setminus e_j$, Breaker will claim another vertex in $e_i \setminus e_j$. Since $e_i$ was arbitrary, Maker cannot claim any edge of $S$.

{\bf Case 5:} There are at least two vertices in $e_0 \cap e_1$ and so Maker cannot claim all of the vertices in $e_0$ or $e_1$ since as soon as she tries to claim one of the vertices in $e_0 \cap e_1$, Breaker will claim another vertex in $e_0 \cap e_1$. 
\end{proof}

{\bf \ref{Q3} A probabilistic lemma.}
\begin{lemma}\label{pp1}
Given $B$ is an $\ell \times k$ matrix, if $\ell$ divides $k-1$, then with high probability $H$ does not contain a bicycle.
\end{lemma} 

\begin{proof}
Let $R_{b_t}$ be a random variable counting the number of bicycles which are in $H$. By Proposition~\ref{Markov}, it suffices to show that the expectation of $R_{b_t}$ converges to zero as $n$ tends to infinity.

We let $R_{b_1},\dots,R_{b_8}$ respectively count the number of hypergraphs $J:=e_{f_1},\dots,e_{f_u}$ in $H$ with $u \leq \log n$, for which $J$ corresponds to
\begin{itemize}
\item[(i)]{a spoiled cycle;}
\item[(ii)]{a double overlapping pair;}
\item[(iii)]{a double loose cycle;}
\item[(iv)]{an overlapping pair with handle;}
\item[(v)]{a loose cycle with handle;}
\item[(vi)]{a loose cycle to loose cycle;}
\item[(vii)]{an overlapping pair to overlapping pair;}
\item[(viii)]{an overlapping pair to loose cycle.}
\end{itemize}
Note that each of these hypergraphs contain a loose path of length $u-2$; hence let $R_{b_9}$ count the number of loose paths $e_{f_1},\dots,e_{f_u}$ in $H$ with $u \geq (\log n)-1$. Then we have $R_{b_t} \leq \sum_{i=1}^9 R_{b_i}$ and hence it suffices to show $\ex (R_{b_i})=o(1)$ for each $i$. The cases for $i=1,5,9$ were covered by R\"odl and Ruci\'nski's proof, however we will repeat them here for clarity.

Suppose that $J:=e_{f_1},\dots,e_{f_u}$ is the valid edge order corresponding to one of the nine cases listed above given by the definitions earlier. When calculating an upper bound on the expected number of copies of some hypergraph $J$ in $H$, we need to first bound the number of ways to \emph{draw} $J$ (i.e.  bound the number of non-isomorphic hypergraphs which $J$ could be - e.g. for a spoiled cycle $e_{f_1},\dots,e_{f_u}$ we need to choose the size of the intersection $e_{f_1} \cap e_{f_2}$, and also the number of edges $u$). Second, we should consider $J$ as being drawn, and bound the number of ways to pick elements from $[n]_p$ to represent each vertex of $J$. Thus we are interested in bounding the number of ways of drawing each $J$ and also the number of ways of choosing representatives from $[n]_p$ for each vertex of $J$.

Each hypergraph $J$ which we wish to count can be written as a union of at most three hypergraphs $P_1,P_2,P_3$, for which each of these are one of an overlapping pair, loose cycle, or loose path.
Further, if $P_2$ and $P_3$ exist, we have $|V(P_1) \cap V(P_2)| \leq k-1$, $|V(P_2) \cap V(P_3)| \leq k-1$ and $|V(P_1) \cap V(P_3)|=\emptyset$. Thus for each $i \in [2]$ we have at most $(|V(P_i)| \cdot |V(P_{i+1})|)^{k-1}$ choices for how to make $P_i$ and $P_{i+1}$ intersect. There is only one way to draw a loose cycle or loose path, and at most $k-2$ ways to draw an overlapping pair. Further for each $J$ in (i)--(viii), we have $|V(J)| \leq k \log n$. Thus the total number of ways of drawing each $J$ in (i)--(viii) is at most polylogarithmic in $n$. 

Recall that $B$ is a strictly balanced matrix of dimension $\ell \times k$, and hence for every $W \subseteq [k]$ for which $2 \leq |W| < k$ we have
\begin{align}\label{pf1}
\frac{|W|-1}{|W|-1+\rank (B_{\overline{W}})-\ell} < \frac{k-1}{k-1-\ell}.
\end{align}
Additionally we have
$m(B)=\frac{k-1}{k-1-\ell}$.
Now let $p < cn^{-1/m(B)}=cn^{-(k-\ell-1)/(k-1)}$ where we choose $c$ to be a constant satisfying
\begin{align}\label{pf3}
c< 1/(ke^2).
\end{align}
Given $i \in [u]$ and $J=e_{f_1},\dots,e_{f_u}$ has been drawn, we wish to bound the expected number of ways of picking $e_{f_i}$ to be an edge with $q$ old elements. Such an edge represents a solution $x$ to $Bx=b'$, where $q$ of the $x_i$ have already been chosen. Let these indices be $W$; we are now attempting to solve $B_{\overline{W}} x'=b''$ for some vector $b''$ of $k-q$ elements. Note also we must choose one of the $q!$ possible assignments of the $q$ indices in $W$ to the $q$ old elements. Thus the expected number of ways, $Y$, of picking the $k-q$ new vertices for $e_i$, satisfies
$$Y \leq \sum_{\stackrel{W \subseteq [k]}{|W|=q}} q! n^{k-q-\rank(B_{\overline{W}})} p^{k-q}.$$
We wish to bound $Y$. By rearranging the inequality given by (\ref{pf1}), we have (if $|W| \geq 2$)
$$\ell (k-|W|) - (k-1) \rank(B_{\overline{W}}) < 0.$$
In fact since all quantities above are integers and $\ell$ divides $k-1$, we must have 
\begin{align}\label{pf5}
\ell (k-|W|) - (k-1) \rank(B_{\overline{W}}) \leq -\ell.
\end{align}
Thus we have 
\begin{align}\label{pf7}
n^{k-|W|-\rank(B_{\overline{W}})} p^{k-|W|} \leq
\begin{cases}
c^{k} n^{\frac{\ell}{k-1}} &\text{if } |W|=0; \\
c^{k-1} &\text{if } |W|=1; \\
cn^{\frac{-\ell}{k-1}} &\text{if } 2 \leq |W| \leq k-1.
\end{cases}
\end{align}
(Note that here we used $\rank(B_{\overline{W}})=\ell$ if $|W|=1$; see Proposition 4.3 in~\cite{HST}.)

For each hypergraph $J$ in (i)--(viii), there is always precisely one initial edge, $u-3$ good edges, and two bad edges. Thus the number of choices we have for picking which element of $[n]_p$ to use for each vertex in $J$ has expectation which is at most:

\begin{align*}
& n^{k-\ell} p^k \left( \sum_{\stackrel{W \subseteq [k]}{|W|=1}} n^{k-|W|-\rank(B_{\overline{W}})} p^{k-|W|} \right)^{u-3} \left( \sum_{\stackrel{W \subseteq [k]}{2 \leq |W| \leq k-1}} |W|! n^{k-|W|-\rank(B_{\overline{W}})} p^{k-|W|} \right)^2 \\
\stackrel{(\ref{pf7})}{\leq} & (k!)^2 \cdot (kc)^u \cdot n^{-\ell/(k-1)}.
\end{align*}
We conclude that for each $i \in [1,8]$ we have $\ex(R_{b_i}) < O( n^{-\ell/(k-1)} \cdot \polylog (n) )=o(1)$.

Finally, note that a loose path with at most $n$ vertices clearly has at most $n$ edges. Further, again using Proposition 4.3 in~\cite{HST}, we have $k \geq \ell+2$. Thus we have
\begin{align*}
\ex(R_{b_9})  
\leq & \, O \left( \sum_{u \geq (\log n) -1}^{n} n^{k-\ell} p^k \prod_{i=1}^{u-1} \left( \sum_{\stackrel{W \subseteq [k]}{|W|=1}} n^{k-|W|-\rank(B_{\overline{W}})} p^{k-|W|} \right) \right) \\
\stackrel{(\ref{pf7})}{\leq} & \, O \left( n^{\frac{\ell}{k-1}} \sum_{u \geq (\log n) -1}^{n} (kc)^u \right) \stackrel{(\ref{pf3})}{=} o(1),
\end{align*}
as required.
\end{proof}

{\bf Putting the parts together.}
To reiterate the main points of the proof of Theorem~\ref{main}(ii), we finish by showing that it follows easily from the lemmas in each of the parts \ref{Q1}--\ref{Q3}.

\begin{proof}[Proof (summary) of Theorem~\ref{main}(ii)]
Let $A$ be a fixed integer-valued matrix of dimension $\ell' \times k'$ and $b$ a fixed integer-valued vector of dimension $\ell'$, such that the pair $(A,b)$ is irredundant, and $A$ is irredundant and satisfies $(*)$. In order to prove w.h.p. Breaker wins the $(A,b)$-game on $[n]_p$, by Proposition~\ref{strictred}, it suffices to show w.h.p. Breaker wins the $(B,b')$-game on $[n]_p$, where $(B,b')$ is the associated pair of $(A,b)$. We rephrase the problem to a game on the hypergraph $H:=H([n]_p,B,b')$. We then show that if a component $H'$ of $H$ 
does not contain a bicycle, then it satisfies certain conditions stated in Lemma~\ref{dl1}. Breaker wins the game played on such a component by Lemma~\ref{dl2}. Supposing that $B$ is an $\ell \times k$ matrix where $\ell$ divides $k-1$, then by Lemma~\ref{pp1}, w.h.p. $H$ (and therefore each component of $H$) does not contain a bicycle. Finally since Breaker can win the game on each component of $H$, he wins the game on $H$, and thus wins the $(A,b)$-game on $[n]_p$ w.h.p., as required.
\end{proof}

\section{Concluding remarks and Proof of Theorem~\ref{singlemain}}\label{secrem}
\subsection{Improvements on Breaker's strategy}
The strange fact that our proof of Theorem~\ref{main}(ii) works when $B(A)$ is an $\ell \times k$ matrix such that $\ell$ divides $k-1$ follows precisely via inequality given by (\ref{pf5}). 
We found an equivalence between bicycles and valid edge orders with two bad edges in Claim~\ref{dlc2}. Suppose we extended our language to \emph{$p$-cycles} (corresponding to valid edge orders with $p$ bad edges), and were able to find a winning strategy for Breaker playing the game on any component of $H$ that does not contain a $p$-cycle. We could then obtain a proof for matrices $A$ for which the associated matrix $B$ of dimension $\ell \times k$ satisfies $\ell=p-1$ (without the need for any divisibility conditions). However given the number of cases that arose from considering bicycles, it would seem unfeasible to attempt this.

For a similar reason we did not consider the alternate problem of allowing Maker to try to obtain solutions to $Ax=b$ which are not $k$-distinct. Allowing solutions with repeats would make the associated hypergraphs non-uniform (e.g. for $x+y=z$ there would be edges of size $2$ and $3$) and therefore require a more in-depth case analysis. Note that Kusch, Ru\'e, Spiegel and Szab\'o do consider the analagous problem in the biased version; see Section~4.3 of~\cite{KRSS}.

Lemma~\ref{dl1} gives a precise description of hypergraphs with at most one bad edge and a fixed number of $k$-bad edges.
What can be said of hypergraphs with at most $p$ bad edges (for fixed $p$) and a fixed number of $k$-bad edges? Also note that our definition of a valid edge order (where every edge after the first one is either good or bad, i.e. there are zero $k$-bad edges) is a hypergraph generalisation of a tree. This is since trees have precisely this property in the graph case; a $2$-bad edge here is an edge which completes a cycle. Thus it would be interesting to obtain a more detailed description of hypergraphs with a valid edge order. 

Observe the following connection of this with R\"odl and Ruci\'nski's proof of the $0$-statement of Theorem~\ref{randomrado}. It is very easy to 2-colour a hypergraph with a valid edge order so that it has no monochromatic edges; simply go through the edges in order, colouring the (at least one) new vertex of an edge $e_i$ the colour which was not assigned to one of the old vertices of $e_i$. Thus if the hypergraph associated to $[n]_p$ has a valid edge order, then $[n]_p$ can be $2$-coloured so that there are no monochromatic solutions to $Ax=0$.

\subsection{Matrices which do not satisfy $(*)$}\label{not*}
Matrices $A$ which are irredundant and do not satisfy $(*)$ traditionally have not received as much attention. For such a matrix, $\mathbb{N}$ is not $(A,0,r)$-Rado for any $r\geq 2$ since we can $2$-colour $\mathbb{N}$ and avoid any monochromatic solutions to $Ax=0$ (see Section 4.1 in~\cite{HST}). Also note that $m(A)$ is ill-defined in this case. Further in the bias version of the $(A,b)$-game, recall that Theorem~\ref{KRSSthm}(ii) states that Breaker wins the (1:2) $(A,b)$-game on $[n]$. All of these facts follow easily via use of the row of the matrix which (under Gaussian elimination) has at most two non-zero entries. We show through some examples that the threshold for the random $(A,b)$-game is at least slightly less trivial.
\begin{thm}\label{notstar}
Let $A$ be a fixed integer-valued matrix of dimension $\ell \times k$ and $b$ a fixed integer-valued vector. Given the pair $(A,b)$ is irredundant and $A$ is irredundant and does not satisfy $(*)$, we have the following:
\begin{itemize}
\item[(i)]{If $A=\begin{pmatrix} \alpha & -\beta \end{pmatrix}$ is such that $\alpha, \beta$ are non-equal positive integers, then Maker wins the $(A,b)$-game on $[n]_p$ w.h.p. if $p \gg n^{-1/3}$.}
\item[(ii)]{Breaker wins the $(A,b)$-game on $[n]_p$ w.h.p. if $p \ll n^{-1/3}$.}
\item[(iii)]{If $A=\begin{pmatrix} \alpha & -\beta & 0 \\ 0 & \alpha & -\beta \end{pmatrix}$ is such that $\alpha, \beta$ are non-equal positive integers, then Breaker wins the $(A,0)$-game on $[n]$ (i.e. the non-biased non-random game; thus with probability equal to one, Breaker wins the $(A,0)$-game on $[n]_p$ for any $0<p<1$).}
\end{itemize}
\end{thm}

\begin{proof}
For the $(A,b)$-game in (i), since $(A,b)$ is irredundant, we have that there exists a solution to $\alpha x_1 - \beta x_2=b$ in $\mathbb{N}$ with $x_1 \not=x_2$. 
Thus we have that $t:=\gcd(\alpha,\beta)$ must divide $b$: hence we may assume without loss of generality that $\gcd(\alpha,\beta)=1$. 

Maker wins the $(A,b)$-game in (i) if there exists a distinct triple $\{(\alpha x-b)/\beta,x,(\beta x+b)/\alpha \} \subseteq [n]_p$, since she can have the first pick and choose $x$. Then she can complete a solution by picking whichever of $(\alpha x-b)/\beta$ and $(\beta x+b)/\alpha$ remains unchosen after Breaker's turn. 

\begin{claim}\label{existz}
There exists a fixed $z \in [0,\alpha \beta-1]$ (depending on $\alpha,\beta,b$) such that whenever $x \equiv z\mod \alpha \beta$, the triple $\{(\alpha x-b)/\beta,x,(\beta x+b)/\alpha \}$ is contained in the integers.
\end{claim}

\begin{proof}
Let $x \in \mathbb{Z}$. Note that we have $(\beta x + b)/\alpha \in \mathbb{Z}$ whenever $\beta x \equiv -b \mod \alpha$. Since $\gcd(\alpha,\beta)=1$, there exists $y \in [0,\alpha-1]$ such that whenever $x \equiv y \mod \alpha$, we have $(\beta x + b)/\alpha \in \mathbb{Z}$. Similarly there exists $y' \in [0,\beta-1]$ such that whenever $x \equiv y' \mod \beta$, we have $(\alpha x-b)/\beta \in \mathbb{Z}$. 
Combining these two facts, the Chinese remainder theorem implies that there
exists $z \in [0,\alpha \beta-1]$ such that whenever $x \equiv z \mod \alpha \beta$ we have $(\beta x + b)/\alpha \in \mathbb{Z}$ and $(\alpha x-b)/\beta \in \mathbb{Z}$.
\end{proof}

Call triples which satisfy the property in Claim~\ref{existz} \emph{good}. From the claim, for sufficiently large $n$ we deduce that $[n]$ contains $n/(2\alpha^2 \beta^2)$ good triples. Further, each $x \in [n]$ is in at most three good triples. Thus, there is a collection $X$ of at least $n/(6 \alpha^2 \beta^2)$ good triples in $[n]$ that are all pairwise disjoint. The expected number of triples in $X$ in $[n]_p$ is $\Theta(np^3)$. Hence if $p \gg n^{-1/3}$ then by Proposition~\ref{Chernoff} w.h.p. there exists $x$ such that $\{(\alpha x-b)/\beta,x,(\beta x+b)/\alpha \} \subseteq [n]_p$, so Maker wins as required.

If $p \ll n^{-1/3}$, then the expected number of triples is $o(1)$, so via Proposition~\ref{Markov} w.h.p. there are no triples of this form at all. For the game in (ii), since $A$ is irredundant but does not satisfy $(*)$, under Gaussian elimination there exists one row of $A$ which consists of $\alpha,-\beta$ (where $\alpha$ and $\beta$ are non-equal positive integers) and zeroes, and thus any solution to $Ax=b$ contains the positive integers $(\beta z+c)/\alpha$ and $(\beta ((\beta z +c)/\alpha)+ c)/\alpha$ for some rational number $z$ and fixed integer $c=c(A,b)$. Since w.h.p. there are no triples (replacing $b$ with $c$ and $x$ with $z$), whenever we have the pair $(\beta z+c)/\alpha, (\beta ((\beta z +c)/\alpha)+ c)/\alpha \in [n]_p$, then $z, (\beta ((\beta ((\beta z +c)/\alpha)+ c)/\alpha)+c)/\alpha \notin [n]_p$. Thus Breaker can devise a pairing strategy to win the game.

For (iii), as in (i) we may assume without loss of generality that $\gcd(\alpha,\beta)=1$. Then every solution in $[n]$ is a triple of the form $\{\alpha^2 x,\alpha \beta x,\beta^2 x\}$ for some $x \in \mathbb{N}$. Every element of $[n]$ which could be in a solution is of the form $\alpha^i \beta^j y$ with $y \in \mathbb{N}$, where $\alpha$ and $\beta$ do not divide $y$, and $i,j$ are non-negative integers where at most one of $i$ or $j$ is zero. 
Using these facts, Breaker has a strategy to win the game. Indeed, Breaker can create a pairing strategy as follows: pair $\alpha^i \beta^j y$ with $\alpha^{i+1} \beta^{j-1} y$ whenever $i$ is even and $j \geq 1$. Then observe that for any triple $\{\alpha^2 x,\alpha \beta x,\beta^2 x\}$ with $x \in \mathbb{N}$, the middle element is paired with one of the two end elements, so Maker cannot obtain a triple.  
\end{proof}

\subsection{Proof of Theorem~\ref{singlemain}}\label{secone}
We conclude the paper by finishing the proof of Theorem~\ref{singlemain}. 

\begin{thm-hand}[~\ref{singlemain}]
Let $A$ be a fixed integer-valued matrix of dimension $1 \times k$ and $b$ a fixed integer (i.e. $Ax=b$ corresponds to a single linear equation $a_1 x_1 + \dots + a_k x_k=b$ with the $a_i$ non-zero integers). 
\begin{itemize}
\item[(i)]{If the pair $(A,b)$ is irredundant and $A$ is irredundant and satisfies $(*)$, then the $(A,b)$-game on $[n]_p$ has a threshold probability of $\Theta(n^{-\frac{k-2}{k-1}})$;}
\item[(ii)]{If the pair $(A,b)$ is irredundant and $A$ is irredundant and does not satisfy $(*)$, then the $(A,b)$-game on $[n]_p$ is Maker's win if $p \gg n^{-1/3}$ and Breaker's win if $p \ll n^{-1/3}$;}
\item[(iii)]{If the pair $(A,b)$ is irredundant and $A$ is not irredundant, then 
\begin{itemize}
\item[(a)]{the $(A,b)$-game on $[n]_p$ is Breaker's win w.h.p. for any $p=o(1)$ if the coefficients $a_i$ are all positive or all negative;}
\item[(b)]{the $(A,b)$-game on $[n]_p$ is Maker's win if $p \gg n^{-1/3}$ and Breaker's win if $p \ll n^{-1/3}$ otherwise;}
\end{itemize}}
\item[(iv)]{If the pair $(A,b)$ is not irredundant, then the $(A,b)$-game on $[n]$ is (trivially) Breaker's win.}
\end{itemize} 
\end{thm-hand} 

\begin{proof}
\begin{itemize}
\item[(i)]{As discussed in the introduction, this follows immediately from Theorem~\ref{main}.}
\item[(ii)]{We have $k \geq 3$ if and only if $A$ satisfies $(*)$. Hence if $A$ does not satisfy $(*)$ and is a linear equation, we must have $k=2$. So write $A=\begin{pmatrix} \alpha & \beta \end{pmatrix}$, where $\alpha, \beta \in \mathbb{Z}$. Note that since $A$ is irredundant there exist $x_1, x_2 \in \mathbb{N}$ such that $\alpha x_1+\beta x_2=0$ where $x_1 \not= x_2$. Thus we must have $\alpha>0$ and $\beta<0$ or vice versa, and $\alpha \not= -\beta$. Thus the result follows from Theorem~\ref{notstar}(i) and (ii). 
}
\item[(iii)]{Note that any linear equation $a_1 x_1 + \dots a_k x_k=0$ with $k \geq 3$ clearly has a $k$-distinct solution in $\mathbb{N}$ if there exists at least one positive $a_i$ and at least one negative $a_j$, for some $i,j \in [k]$. The same holds for $k=2$ unless if $a_1=-a_2$. Thus, since $A$ is not irredundant, we have one of the following: 
\begin{itemize}
\item[(a)]{the $a_i$ are all positive integers or all negative integers;}
\item[(b)]{we have $k=2$ and $a_1=-a_2$.}
\end{itemize}
For (a), we may assume without loss of generality that $a_1,\dots,a_k$ and therefore $b$ are positive integers. For such a game there are a finite number of $k$-distinct solutions in $\mathbb{N}$, all of which are contained in $[b]$. Thus for any $p=o(1)$, w.h.p. there are no solutions in $[n]_p$ by Proposition~\ref{Markov}, so the game is Breaker's win.
For (b), the existence of any triple $\{x-b/a_1,x,x+b/a_1\}$ leads to a win for Maker, meanwhile if no triples exist then Breaker can win by a pairing strategy. Since the number of such triples in $[n]_p$ is of order $np^3$, the result follows by a similar argument to that given for Theorems~\ref{notstar}(i) and (ii).
}
\item[(iv)]{The $(A,b)$-game is trivially Breaker's win, since there are no winning sets in $\mathbb{N}$.}
\end{itemize}
\end{proof}

\section*{Acknowledgements}
The author would like to thank Christoph Spiegel for helpful conversations on~\cite{KRSS}.  The author is also grateful to Andrew Treglown for his general support and encouragement for writing this paper, and for reading the manuscript. Finally the author thanks the anonymous referees for their helpful comments and suggestions.

{\footnotesize \obeylines \parindent=0pt

Robert Hancock
Institute of Mathematics
Czech Academy of Sciences
\v{Z}itn\'a 25
110 00
Praha
Czechia
}
\begin{flushleft}
{\emph{E-mail address}:
\tt{hancock@math.cas.cz}}
\end{flushleft}


\begin{thebibliography}{99}

\bibitem{BMS} J. Balogh, R. Morris and W. Samotij, Independent sets in hypergraphs, \emph{J. Amer. Math. Soc.}  {\bf 28} (2015),  669--709.


\bibitem{Beck1} J. Beck, Combinatorial Games: Tic-Tac-Toe Theory. \emph{Cambridge University Press} (2008).

\bibitem{Beck2} J. Beck, Van der Waerden and Ramsey type games, \emph{Combinatorica} {\bf 1} (2) (1981), 103--116.  

\bibitem{BL} M. Bednarska and T. \L uczak, Biased positional games for which random strategies are nearly optimal, \emph{Combinatorica} {\bf 20} (4) (2000), 477--488.  


\bibitem{BT} B. Bollob\'as and A. Thomason, Threshold functions, \emph{Combinatorica} {\bf 7} (1), (1987), 35--38. 

\bibitem{CG} D. Conlon and W. T. Gowers, Combinatorial theorems in sparse random sets,  \emph{Ann. Math.} {\bf 84} (2016), 367--454.

\bibitem{ES} P. Erd\H os and J. L. Selfridge, On a combinatorial game, \emph{J. Comb. Theory Ser. A} {\bf 14} (1973), 298--301.


\bibitem{FRS} E. Friedgut, V. R\"odl and M. Schacht, Ramsey properties of random discrete structures, \emph{Random Structures \& Algorithms} {\bf 37} (2010), 407--436.


\bibitem{Hanc} R. Hancock, \emph{PhD thesis, University of Birmingham}, submitted.

\bibitem{HST} R. Hancock, K. Staden and A. Treglown, Independent sets in hypergraphs and Ramsey properties of graphs and the integers, to appear in \emph{SIAM J. Discrete Math.}

\bibitem{HKSS} D. Hefetz, M. Krivelevich, M. Stojakovi\'c and T. Szab\'o, Positional Games,  \emph{Oberwolfach Seminars}, vol. 44, Birkh\" auser (2014).

\bibitem{HL} N. Hindman and I. Leader, Nonconstant monochromatic solutions to systems of linear equations, \emph{Topics in Discrete Math. Springer, Berlin} (2006), 145--154.




\bibitem{KSV} D. Kr\'al', O. Serra and L. Vena, A removal lemma for systems of linear equations over finite fields, {\em Israel J.  Math.}, 187, (2012), 193--207.


\bibitem{KRSS} C. Kusch, J. Ru\'e, C. Spiegel and T. Szab\'o, On the optimality of the uniform random strategy, to appear in \emph{Random Structures \& Algorithms}.



\bibitem{NSS} R. Nenadov, A. Steger and M. Stojakovi\'c, On the threshold for the Maker-Breaker $H$-game, \emph{Random Structures \& Algorithms}, {\bf 49} (3) (2016), 558--578.


\bibitem{Rado} R. Rado, Studien zur kombinatorik, \emph{Mathematische Zeitschrift}, {\bf 36} (1933), 424--470.

\bibitem{RR1} V. R\"odl and A. Ruci\'nski, Lower bounds on probability thresholds for Ramsey properties, in Combinatorics, Paul Erd\H os is Eighty, Vol. 1, 317--346, Bolyai Soc. Math. Studies, J\'anos Bolyai Math. Soc., Budapest, 1993.

\bibitem{RR2} V. R\"odl and A. Ruci\'nski, Random graphs with monochromatic triangles in every edge coloring, \emph{Random Structures \& Algorithms} {\bf 5} (1994), 253--270.

\bibitem{RR3} V. R\"odl and A. Ruci\'nski, Threshold functions for Ramsey properties, \emph{J. Amer. Math. Soc.} {\bf 8} (1995), 917--942. 

\bibitem{RR4} V. R\"odl and A. Ruci\'nski, Rado partition theorem for random subsets of integers, \emph{Proc. London Math. Soc.} {\bf 74} (3) (1997), 481--502. 

\bibitem{ST} D. Saxton and A. Thomason, Hypergraph containers, \emph{Invent. Math.} {\bf 201} (2015), 925--992.

\bibitem{Sch} M. Schacht, Extremal results for random discrete structures, \emph{Ann. Math.} {\bf 184} (2016), 331--363. 

\bibitem{Spie} C. Spiegel, A Note on Sparse Supersaturation and Extremal Results for Linear Homogeneous Systems, \emph{Electron. J. Combin.} {\bf 24} (3) (2017), \#P3.38.


\bibitem{SS} M. Stojakovi\'c and T. Szab\'o, Positional games on random graphs, \emph{Random Structures \& Algorithms}, {\bf 26} (2005), 204--223.






\end{thebibliography}
\end{document}